\documentclass[a4paper,11pt]{amsart}
\usepackage[T1]{fontenc}
\usepackage[latin9]{inputenc}
\usepackage{graphicx}
\usepackage{epsfig}
\usepackage{amssymb}
\usepackage{amsmath}
\usepackage{amsthm}
\usepackage[all]{xy}
\usepackage{verbatim}
\usepackage{url}

\usepackage{comment}

\newenvironment{commentM}{\begingroup \sffamily [{\bf Mikhail:\ }}{]\endgroup}
\def\bM{\begin{commentM}}
\def\eM{\end{commentM}\ }

\theoremstyle{plain}
\newtheorem{lemma}{Lemma} [section]
\newtheorem{proposition}[lemma]{Proposition}
\newtheorem{theorem}[lemma]{Theorem}
\newtheorem{corollary}[lemma]{Corollary}

\theoremstyle{definition}

\newtheorem{example}[lemma]{Example}

\newtheorem{definition}[lemma]{Definition}

\newtheorem{construction}[lemma]{Construction}
\newtheorem{notation}[lemma]{Notation}

\theoremstyle{definition}

\newtheorem{remark}[lemma]{Remark}

\numberwithin{equation}{section}

\begin{document}


\def\noi{\noindent}
\def\R{{\mathbb{R}}}
\def\Z{{\mathbb{Z}}}
\def\C{{\mathbb{C}}}
\def\Q{{\mathbb{Q}}}
\def\F{{\mathbb{F}}}
\def\Xt{{\tilde{X}}}

\def\id{{{\rm id}}}
\def\diag{{\rm diag}}
\def\gcd{{{\rm gcd}}}
\def\lcm{{{\rm lcm}}}
\newcommand{\Img}{\operatorname{Im}}
\def\Hom{{\rm Hom}}
\def\Inn{{\rm Inn}}
\def\Aut{{\rm Aut}}
\def\Lie{{\rm Lie\,}}
\def\coker{{\rm coker\,}}
\def\tors{{\rm tors}}
\def\Ext{{\rm Ext}}
\def\Stab{{\rm Stab}}
\def\Ad{{\rm Ad}}
\def\ad{{\rm ad}}
\def\ab{{\rm ab}}

\def\AA{{\mathbf{A}}}
\def\BB{{\mathbf{B}}}
\def\CC{{\mathbf{C}}}
\def\DD{{\mathbf{D}}}
\def\EE{{\mathbf{E}}}
\def\FF{{\mathbf{F}}}
\def\GG{{\mathbf{G}}}
\def\XX{{\mathbf{X}}}
\def\YY{{\mathbf{Y}}}
\def\ZZ{{\mathbf{Z}}}
\def\VV{{\mathbf{V}}}
\def\WW{{\mathbf{W}}}
\newcommand{\BST}{\boldsymbol{\Theta}}
\newcommand{\TD}{\mathbf{A}_{n-1}^{(m,n-1)}}


\def\into{\hookrightarrow}
\newcommand{\isoto}{\overset{\sim}{\to}}

\newcommand{\Orbset}{\operatorname{Cl}}         
\def\Or{{\rm Cl}}                                
\newcommand{\Orbs}[1]{ \# \mathrm{Cl}( {#1} ) } 
\newcommand{\Wsim}[1]{ \underset{{#1}}{\sim}  }   

\def\Cl{{\rm Cl}}                                
\newcommand{\Cls}[1]{ \# \mathrm{Cl}( {#1} ) } 

\def\a{{\boldsymbol{a}}}
\def\aa{{\boldsymbol{a}}}
\def\bb{{\boldsymbol{b}}}
\def\cc{{\boldsymbol{c}}}
\def\dd{{\boldsymbol{d}}}
\def\ttt{{\boldsymbol{t}}}
\def\sss{{\boldsymbol{s}}}
\def\mmu{{\boldsymbol{\mu}}}
\def\nnu{{\boldsymbol{\nu}}}

\def\vk{{\varkappa}}

\def\zbar{{\overline{z}}}
\def\gbar{{\overline{g}}}
\def\nbar{{\overline{n}}}

\newcommand{\labelto}[1]{\xrightarrow{\makebox[1.5em]{\scriptsize ${#1}$}}}

\newcommand{\redu}{\mathrm{red}}
\newcommand{\sem}{\mathrm{ss}}
\newcommand{\scon}{\mathrm{sc}}
\newcommand{\tor}{\mathrm{tor}}

\newcommand{\GL}{{\bf{GL}}}
\newcommand{\SL}{{\bf{SL}}}
\newcommand{\Sp}{{\bf{Sp}}}
\newcommand{\PSp}{{\bf{PSp}}}
\newcommand{\SO}{{{\bf SO}}}
\newcommand{\PSO}{{\bf{PSO}}}
\newcommand{\Spin}{{{\bf Spin}}}
\newcommand{\HSpin}{{\bf{HSpin}}}
\newcommand{\PGL}{{\bf{PGL}}}
\newcommand{\SU}{{\bf SU}}
\newcommand{\PSU}{{\bf PSU}}

\newcommand{\boxone}{ *+[F]{1} }
\def\ccc{{  \lower0.20ex\hbox{{\text{\Large$\circ$}}}}}
\def\Lbul{{\lower0.20ex\hbox{\text{\Large$\bullet$}}}}
\newcommand{\bc}[1]{{\overset{#1}{\ccc}}}
\newcommand{\bcu}[1]{{\underset{#1}{\ccc}}}
\newcommand{\bcb}[1]{{\overset{#1}{\Lbul}}}
\newcommand{\bcbu}[1]{{\underset{#1}{\Lbul}}}
\newcommand{\sxymatrix}[1]{ \xymatrix@1@R=5pt@C=9pt{#1} }
\newcommand{\mxymatrix}[1]{ \xymatrix@1@R=0pt@C=9pt{#1} }
\newcommand{\rline}{ \ar@{-}[r] }
\newcommand{\lline}{ \ar@{-}[l] }
\newcommand{\dline}{ \ar@{-}[d] }
\newcommand{\upline}{ \ar@{-}[u] }
\def\arr{\ar@{=>}[r]}
\def\all{\ar@{=>}[l]}

\def\RRR{{\Rightarrow}}
\def\gr{\! > \!}
\def\less{ \!\! < \!\! }
\def\LLL{\!\Leftarrow\!}
\def\boe{{\sxymatrix{\boxone}}}
\def\ll{\!-\!}

\def\ds{\displaystyle}

\newcommand\two[2]{\underset{\ds#2}{#1}}
\newcommand\three[3]{\underset{\ds#3}{\two{#1}{#2}}}
\newcommand\four[4]{\underset{\ds#4}{\three{#1}{#2}{#3}}}
\newcommand\five[5]{\underset{\ds#5}{\four{#1}{#2}{#3}{#4}}}


\def\gSU{{\bf SU}}
\def\AA{{\bf A}}

\def\nbar{{\bar{n}}}

\def\bb{{\boldsymbol{b}}}

\def\Htil{{\widetilde{H}}}
\def\Ttil{{\widetilde{T}}}
\def\im{{\rm im\,}}

\def\Gtil{{\widetilde{G}}}
\def\ytil{{\tilde t}}
\def\ttil{{\tilde t}}
\def\ii{{\mathbf{i}}}

\def\emm{\bfseries}

\def\Gal{{\rm Gal}}
\def\H{{\mathbb{H}}}

\def\ppi{{\boldsymbol{\pi}}}

\def\jj{{\boldsymbol{j}}}
\def\ii{{\boldsymbol{i}}}

\def\OO{{\mathbf{O}}}
\def\ve{\varepsilon}

\def\sV{{\mathcal{V}}}
\def\xbar{{\bar{x}}}
\def\U{{\bf U}}
\def\sH{{\mathcal{H}}}

\def\vev{\varepsilon^\vee}
\def\onto{\twoheadrightarrow}

\def\HR{{\bf HR}}
\def\HC{{\bf HC}}
\def\RF{{\bf RF}}
\def\HH{{\bf HH}}

\def\gg{{\mathfrak{g}}}

\def\T{{\mathcal{T}}}
\newcommand{\X}{{{\sf X}}}
\def\G{{\mathbb{G}}}

\def\half{{\tfrac{1}{2}}}

\def\Ss{\sideset{}{'}\sum_{k\succeq i}}
\def\Sst{\Ss}
\def\Ssd{\sideset{}{'}\sum_{k\succeq i,\,k\in D^\tau}}

\newcommand{\Reps}{\textbf{Representatives of classes }}
\newcommand{\reps}{\textbf{representatives of classes }}
\newcommand{\creps}{\textbf{canomical representatives of classes }}
\newcommand{\minreps}{\textbf{minimal representatives of classes }}
\newcommand{\Zerois}{\textbf{The class of zero} consists of }
\newcommand{\zerois}{\textbf{the class of zero} consists of }

\newcommand{\boBE}{ This subsection is based on our calculations in the paper \cite{BE-real} by Borovoi and me.}
\newcommand{\expBE}{ This section is an expanded version of the calculations we did in \cite{BE-real}. }
\newcommand{\ourpaper}{{\em This subsection is based on our paper, \cite{BE-real}.}}

\def\hs{\kern 0.8pt}
\def\hh{\kern 1.0pt}
\def\kk{\kern 2.0pt}



\date{\today}

\title{Solutions of Reeder's Puzzle}

\author{Izhak (Zachi) Evenor}

\begin{abstract}
In this paper we consider the generalized Reeder's puzzle, introduced by Reeder in 2005 and generalized by Borovoi and Evenor in 2016. We give a detailed solution of the puzzle for the graphs of Dynkin diagrams and affine Dynkin diagrams. We find the number of equivalence classes in each case. We also discuss more general graphs, and prove the main theorem about graphs (simply-laced trees) that contain $\EE_6$ as a subgraph.
\end{abstract}

\address{Evenor: Raymond and Beverly Sackler School of Mathematical Sciences,
Tel Aviv University, 6997801 Tel Aviv, Israel}
\email{zachi.evenor@gmail.com}


\maketitle




\setcounter{section}{-1}

\section{Introduction}

Our paper begin with the following game, first introduced by Reeder in 2005\cite{Reeder}, and generalized in 2014/2015 by \cite{BE-real}. The game, which you can play at the beach, is as follows: draw on the sand a connected graph with $n$ vertices (draw each vertex as an empty circle). On every vertex we can put one seashell, and thus for each vertex there can at two states: it is either empty or shelled. There can be $2^n$ configurations of seashells on the graph. Now we define the allowed moves in the game: choose any vertex, if it has an even number of neighbors with shells on them -- do nothing, if it has an odd number of neighbors with shells on them -- change the vertex's status (if it is empty - put a shell, if it has a shell - remove it). The goal of the game is to take any configuration and transform it to a configuration with minimal number of shells using only the allowed moves (which we call simply {\em moves}). A solution of the game is the number of equivalence classes (the moves induce an equivalence relation, see below) and for each class to find a representative labeling with minimal number of shells (see Definition \ref{def:goal}). We call this game {\em Reeder's Game} or {\em Reeder's Puzzle}.

It is easy to translate this game to a mathematical language: let $D$ be a connected graph with $n$ vertices, we say that two vertices $i,j$ of a diagram $D$  are neighbors if they are connected by an edge. Each vertex can be labeled with two values: 0 or 1. In other words, a vertex $i$ has the value $a_i \in \{0,1\} = \mathbb{Z}/2\mathbb{Z}$. A configuration is a vector of labels $\aa = (a_1,...,a_n) = (a_i)_{i=1}^n \in (\mathbb{Z}/2\mathbb{Z})^n$, which we call a {\em labeling}. The allowed moves can also be formulated as mathematical transformations $\T_i : (\mathbb{Z}/2\mathbb{Z})^n \to (\mathbb{Z}/2\mathbb{Z})^n$, $\T_i(\aa)=\aa'$. See Section \ref{sec:rules}. In section \ref{sec:basics} we list some basic observations that will help us solve the puzzle.

\bigskip

The following motivating cases use the solution of Reeder's puzzle. One can skip this part without harming the understanding of the rest of this text.

The original Reeder's puzzle arose in \cite{Reeder} to study orthogonal linear maps on a vector space $V$ over  $\mathbb{F}_2$ that preserve the quadratic
\begin{equation}
q(u) = \sum_i u_i^2 + \sum_{i \ll j} u_i u_j = \# \mbox{ of 1's} \ + \ \# \mbox{ edges between 1's} \ .
\end{equation}

In particular, Reeder studies the diagram to determine when the map
\begin{equation*}
\rho : W \to \mathrm{O}(V,q)
\end{equation*}
between the Coxeter group $W$ generated by the moves on the diagram and the orthogonal subgroup $\mathrm{O}(V,q) \subset \mathrm{GL}_n(\mathbb{F}_2)$ of linear maps preserving $q$, is surjective.

The generalized Reeder's puzzle was used in \cite{BE-real} to compute the Galois cohomology (see \cite{Serre} and \cite{Be}) of simply-connected reductive linear algebraic groups $H^1(\mathbb{R},G)$, and with appropriate modification also handle non-simply-laced Dynkin diagrams, based on \cite{Bo} and \cite{Borovoi-arXiv}. Moreover, one can use the method of \cite{BE-real} to compute the Galois cohomology for (non-compact) inner and outer forms of compact simply-connected linear algebraic groups by considering a twisted action of the Weyl group, which modify a little the rules of the game. These applications are both handled with details in my paper with Borovoi.\cite{BE-real} We used Reeder's puzzle to calculate the cardinalities of the Galois cohomology sets (computed in another method by Adams, see \cite{A}) and also developed a method to find the number of connected components of the algebraic variety $(G/F)(\mathbb{R})$ where $F \subset G$ are simply-connected linear algebraic groups, using the cohomology sets and their descriptions by Reeder's puzzle.

Kac-Moody groups are infinite-dimensional groups with Kac-Moody algebras as their tangent space. An affine Dynkin diagram encodes a generalized Cartan matrix which characterizes a Kac-Moody algebra through root data and thus a Kac-Moody group.\cite{KP}\cite{KW}\cite[mainly sections 1 and 5]{Tits} Thus the Reeder's puzzle solutions for affine Dynkin diagrams are related to the conjugacy classes of elements of order dividing 2 of those infinite-dimensional groups.


\smallskip

Reeder's puzzle is similar but not identical to the lit-only $\sigma$-game. Both puzzles are based on a two-states configurations on connected graphs but the allowed moves are different in each game. The allowed moves in the lit-only $\sigma$-game are to choose a vertex with label 1, (i.e. vertex $i$ such that $a_i=1$) and then flip the states ($0 \leftrightarrow 1$) of all of its neighbors. \cite{Huang-Weng-flipping-puzzle-2010}\cite{Huang-lit-only-s-game-2015} The lit-only $\sigma$-game was solved for Vogan diagrams, Dynkin diagrams and extended Dynkin diagrams by Meng-Kiat Chuah and Chu-Chin Hu.\cite{Chuah-Hu-2004}\cite{Chuah-Hu-2006}

It turns out that Reeder's puzzle and the lit-only $\sigma$-game are related to each other: they are dual in the sense of \cite{Huang-lit-only-and-dual-2009}. To every lit-only $\sigma$-game move there is a Reeder's puzzle move, and when the moves are expressed as $n \times n$ matrices operating on $V = \mathbb{F}_2^n = (\mathbb{Z}/2\mathbb{Z})^n$ the correspondence rule is $\T_i = S_i^t$ where $\T_i$ is the Reeder's puzzle move on vertex $i$ expressed as matrix and $S_i$ is the lit-only $\sigma$-game move expressed as matrix. Let $A$ be the adjacency matrix of the graph $D$: $a_{ij}=1$ if the vertices $i$ and $j$ are neighbors and 0 otherwise. Then by \cite[Theorem 2.1]{Huang-lit-only-and-dual-2009} for every $1 \le i \le n$, $A\T_i = A S_i^t = S_i A$. When $\det A = 1$ this gives a one-to-one correspondence between Reeder equivalence classes and lit-only $\sigma$-game orbits:
\begin{eqnarray*}
\Cl(D) & \longrightarrow & \mathrm{Or}(D) \\
 \rho & \longmapsto & A\rho
\end{eqnarray*}
where $\Cl(D)$ is the set of Reeder classes and $\mathrm{Or}(D)$ is the set of lit-only $\sigma$-game orbits.\cite[Corollary 2.2]{Huang-lit-only-and-dual-2009} Therefor by obtaining results about one game can obtain results about the other and in particular the results presented in our paper can help solving lit-only $\sigma$-games.

\bigskip

The rest of the paper is structured as follows: In Section \ref{sec:rules} we describe mathematically the rules of the game for simply-laced and non-simply-laced graphs, and also discuss several basic observations which are essential for solving the game. In section \ref{sec:Calculations} we solve it for connected Dynkin and affine Dynkin diagrams (taken from \cite[Table 1 and Table 6]{OV}). In \cite{BE-real} we solved it for compact and ``twisted'' Dynkin diagrams. We give detailed derivations of the results for  $\AA_n$, $\BB_n$, $\CC_n$, $\DD_n$. The results for $\EE_6$, $\EE_7$, $\EE_8$, $\FF_4$ and $\GG_2$ can be found in \cite{BE-real}.
The results for $\EE_6$, $\EE_7$ and $\EE_8$ were obtained and proved independently by Reeder.\cite{Reeder}
For $\AA_n$ we also add important definitions, lemmas and propositions (which appear in \cite{BE-real}) that are vital for understanding the rest of the calculations. In Section \ref{sec:general-graphs} we discuss general results for more general graphs.


\section{Rules of the Game}
\label{sec:rules}

\subsection{Simply-laced diagrams}

Let $D$ be a connected simply-laced graph with $n$ vertices. We say that two vertices $i,j$ of a diagram $D$  are neighbors if they are connected by an edge.

\begin{definition}
To each vertex we attach a label $a_i \in \{0,1\} = \mathbb{Z}/2\mathbb{Z}$. A vector of labels $\aa = (a_i)_{i=1,...n} \in (\mathbb{Z}/2\mathbb{Z})^n$ is called a {\em labeling} of the diagram.
\end{definition}

\begin{definition}[Moves] \label{def:elem-moves}
A {\em move} $\T_i$ is a map sending a labaling $\aa$ to another labeling $\aa'$ defined as follows
\begin{equation} \label{eq:moves-simple}
\aa' = \T_i(\aa) \quad \mbox{ where } \begin{cases} a'_j = a_j & \mbox{ if } j \ne i \ , \\ a'_i = a_i + \sum_k a_k & \mbox{ if } j=i \ , \end{cases}
\end{equation}
where $k$ runs on the neighbors of the vertex $i$, and the summation is done modulo 2. Note that the move $\T_i$ can change only the vertex $i$.
\end{definition}

\begin{remark} \label{rem:elem-move}
Regarding Definition \ref{def:elem-moves}, the move $\T_i$ can be described as follows: if $a_i$ has an odd number of neighbors with 1 -- change it, if $a_i$ has an even number of neighbors with 1 -- do not change it. Cf. introduction section.
\end{remark}

\begin{proposition} \label{prop:elem-move}
The moves satisfy $\T_i^2 = \id$, that is $\T_i(\T_i(\aa)) = \aa$.
\end{proposition}
\begin{proof}
It is clear from Remark \ref{rem:elem-move}.
\end{proof}

\begin{example}
Consider the diagram $\AA_3$: $\sxymatrix{ \bc{1} \rline & \bc{2} \rline & \bc{3} }$.
Then:
\begin{eqnarray*}
1 \ll \boldsymbol{1} \ll 1 \overset{\T_2}{\longmapsto} 1 \ll \boldsymbol{1} \ll 1 \ , & & a'_2 = a_1 + a_2 + a_3 = 1 + 1 + 1 = 1 , \\
\boldsymbol{1} \ll 1 \ll 1 \overset{\T_1}{\longmapsto} \boldsymbol{0} \ll 1 \ll 1 \ ,  & & a''_1 = a_1 + a_2 = 1 + 1 = 0 .
\end{eqnarray*}
\end{example}

\begin{remark}
One can think on a labeling as a vector in $V = \mathbb{F}_2^n$ and then formulate each move as a matrix operation on these vectors. See \cite{Huang-lit-only-and-dual-2009} for details.
\end{remark}

\begin{definition}
Two labeling $\aa$ and $\aa'$ are equivalent if there is a finite sequence of moves $\T_{i_1},...,\T_{i_r}$ such that $\aa' = \T_{i_1} \circ \cdots \circ \T_{i_r}(\aa) = \T_{i_1}(...(\T_{i_r}(\aa)...)$, i.e. $\aa$ can be changed to $\aa'$ by a finite sequence of moves. It can be easily checked that this is an equivalence relation (use Proposition \ref{prop:elem-move}). In that case we say that $\aa$ and $\aa'$ lie in the same {\em equivalence class} (or just {\em class}) and denote the equivalence class of $\aa$ by $[\aa]$.
\end{definition}

\begin{notation}
Let $D$ be a connected diagram. We denote the set of labelings on the diagram by $L(D)$. We denote the set of equivalence classes by $\Cl(D) :=  L(D)/\sim$ and the number of classes by $\Cls{D}$.
\end{notation}

\begin{definition}[Reeder's Puzzle] \label{def:goal}
{\em Reeder's puzzle} or {\em Reeder's game} consists of a finite connected diagram $D$ with $n$ vertices, the set of all labeling $L(D) \simeq (\mathbb{Z}/2\mathbb{Z})^n $ with the set of moves defined in Definition \ref{def:elem-moves} and its generalization Definition \ref{def:elem-moves-nonsimple}. The goal of the game is to find the equivalence classes under the above defined relation. A solution of the game for a certain diagram $D$ consists of:
\begin{enumerate}
\item Description of the set of the equivalence classes $\Cl(D)$
\item The number of equivalence classes $\Cls{D}$.
\item For each equivalence class we give a representative with a minimal number of 1's, we call such representative a {\em minimal representative}.
\item In certain cases we give a complete description of the equivalence classes.
\end{enumerate}
\end{definition}

\begin{remark} \label{rem:connected}
One may ask why we consider only connected graphs. Suppose a graph $D$ is composed from two disjoint connected subgraphs, $D = D_1 \sqcup D_2$. Then $L(D) \cong L(D_1) \times L(D_2)$ and $\Cl(D) \cong \Cl(D_1) \times \Cl(D_2)$, so we are reduced to the case of connected graphs.
\end{remark}

\subsection{Non-simply-laced diagrams}
\label{subsec:non-simply-laced}

Now let $D$ be a connected graph with $n$ vertices, not necessarily simply-laced. Now we allow also multiple edges. We consider multiple edges that are directed (i.e. they are arrows). When we write
\begin{equation*}
\cdots - a_i \overset{s}{\Rightarrow} a_{i+1} - \cdots \quad \mbox{ for } \quad \sxymatrix{ \cdots \rline & \bc{i} \arr^{s} & \bc{i+1} \rline & \cdots }
\end{equation*}
we mean that the vertices $i$ and $i+1$ are connected by directed edge of multiplicity $s$ and we call the $i$-th (resp. $i+1$-vertex) vertex the {\em longer vertex} (resp. {\em shorter vertex}). In other words, the arrow points on the shorter vertex. When the edge is simply a double edge (i.e. $s=2$) we omit the $s=2$ mark and just write $a_i \RRR a_{i+1}$. When the edge is triple, we write $a_i \Rrightarrow a_{i+1}$ (arrow with 3 lines). We say that two vertices $i,j$ of a diagram $D$  are neighbors if they are connected by at least one edge (single or multiple). Due to Remark \ref{rem:connected} we consider only connected graphs.

\medskip

Motivated by \cite{BE-real} we modify the rules of the game:

\begin{definition}[Moves] \label{def:elem-moves-nonsimple}
For a non-simply-laced diagram we define elementary transformations or {\em moves}. A move $\T_i$ is a map sending a labaling $\aa$ to another labeling $\aa'$ defined as follows
\begin{equation} \label{eq:move-nonsimple}
\aa' = \T_i(\aa) \quad \mbox{ where } \begin{cases} a'_j = a_j & \mbox{ if } j \ne i \ , \\ a'_i = a_i + \sum_{k\ge} a_k & \mbox{ if } j=i \ , \end{cases}
\end{equation}
where $k$ runs on the neighbors of the vertex $i$ \textbf{excluding} shorter neighbors with even multiple edges, and the summation is done modulo 2. As in the previous case, the move $\T_i$ can change only the vertex $i$.
\end{definition}

Note that this definition is a generalization of Definition \ref{def:elem-moves} and \eqref{eq:move-nonsimple} agrees with \eqref{eq:moves-simple} for simply-laced diagrams.

\begin{example}
Let us give an example that clarifies \eqref{eq:move-nonsimple}. Assume we have 2 vertices that are connected by a double edge (assume for simplicity that there are no other neighbors), and the vertex $i$ is longer than the vertex $i+1$. The diagram is $a_i \RRR a_{i+1}$. Then
\[
\aa' = \T_i(\aa) : \quad a'_i = a_i , \qquad \aa' = \T_{i+1}(\aa) : \quad a'_{i+1} = a_i + a_{i+1} \ .
\]
If there are 2 vertices connected by a triple edge, and the vertex $i$ is longer than the vertex $i+1$, then the diagram is $b_i \Rrightarrow b_{i+1}$ and
\[
\bb' = \T_i(\bb) : \quad b'_i = b_i + b_{i+1} , \qquad \bb' = \T_{i+1}(\bb) : \quad b'_{i+1} = b_i + b_{i+1} \ .
\]
\end{example}

\begin{remark} \label{rem:elem-move-non-simply-laced}
If the multiplicity of the directed edge is even then the shorter vertex does not affect the longer vertex. In a more prosaic way we can say that in this case ``The longer vertex does not see the shorter vertex''. Moreover, since the only thing matters is the parity of the multiplicity of the edge, we may consider only single-edged and double-edged graphs.
\end{remark}

The goal of the game remains the same.

\subsection{Basic observations}
\label{sec:basics}


\begin{definition}
By {\em (connected) components (of 1's)} of a labeling of a Dynkin diagram we mean the connected
components of the graph obtained by removing the vertices with zeros and the corresponding edges.
\end{definition}

For example, the following labeling of $\AA_9$ has 3 connected components:
\[
\sxymatrix{ 1 \rline & 1  \rline & 0  \rline & 1  \rline & 0  \rline & 0  \rline & 1  \rline & 1 \rline & 1.}
\]
For some diagrams $D$ the number of of components of a labeling is an invariant.
For some others, the parity of the number of components is an invariant.

\begin{definition}
A {\em fixed labeling} is a labeling that is fixed under all the moves $\T_i$.
\end{definition}

For example, the following labelings of $\AA_5$,
\[
\sxymatrix{ 0 \rline & 0  \rline & 0  \rline & 0  \rline & 0} \quad \text{and}\quad   \sxymatrix{ 1 \rline & 0  \rline & 1  \rline & 0  \rline & 1} \ ,
\]
are fixed.

We say that $i$ is a vertex of degree $d$ if it has exactly $d$ neighbors. We are especially
interested in vertices of degrees 1,2 and 3.

\begin{lemma} \label{lem:invartiant-comp}
Let $D$ be a diagram. Let $i$ be a vertex of degree 1 or 2. Then the move $\T_i$ does not change the number of components.
\end{lemma}
\begin{proof}
The proof consists of easy direct checks.

Let $a_1\ll a_2\ll\cdots$ be a diagram with $k$ components and consider the vertex 1 with label $a_1$. It has one neighbor, namely vertex 2. Then
\[
\T_1(00...) = 00... , \quad \T_1(01...) = 11... , \quad \T_1(10...) = 10... , \quad \T_1(11...) = 01...
\]
and we see that the number of components is preserved.

Let $\cdots \ll a_{i-1} \ll a_i \ll a_{i+1} \ll \cdots$ be diagram with $k$ components and consider the vertex $i$ with label $a_i$. It has two neighbors, namely vertices $i-1$ and $i+1$. Then
\[
\T_i(...0a_i0...) = (...0a_i0...), \quad \T_i(...100...)=...110... , \quad T_i(...110...) = ...100...
\]
and the interesting cases
\[
\T_i(...1a_i1...) = ...1a_i1... \quad \mbox{ since } \quad a'_i = a_i + 1 + 1 = a_i
\]
and we see in all the cases that the number of components is preserved.
\end{proof}

\begin{definition} \label{def:loop}
In this paper a {\em cycle} or a {\em circuit} is a sequence of vertices and edges $(j_0,h_0,j_1,h_1,...,h_r, j_r)$ such that $j_r = j_0$ and each edge $h_i$ appears only once in the sequence (i.e. $h_p \ne h_q$ for all $0 \le p < q \le r$). Intuitively, a circuit is a sequence of vertices connected by edges such that we can go from one of them (say vertex $j_0$) through the others (walking from vertex to other vertex is done by walking through the connecting edge between them) and end in the vertex in which we started from without going on the same edge twice.
For example:
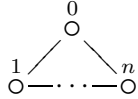
\begin{equation*}
\sxymatrix{ & \bc{0} & \\ \bc{1} \rline \ar@{-}[ru] & \cdots \rline & \ar@{-}[lu] \bc{n} }
\end{equation*}
is a cycle: $(0 \rightarrow 1  \rightarrow ...  \rightarrow n  \rightarrow 0)$.
A diagram is called {\em acyclic} or {\em circuitless} if it has no circuits. If $D$ is acyclic and connected then $D$ is a {\em tree}.
\end{definition}

The Dynkin diagrams $\DD_n$ $(n\ge4)$, $\EE_6$, $\EE_7$ and $\EE_8$ have vertices of degree 3. Now let $D$ be a tree with a vertex $i$ of degree 3, and let $\a$ be a labeling of $D$ that looks as
\begin{eqnarray} \label{eq:unsplit-scheme}
\sxymatrix{\dots \rline  &1 \rline & 1 \rline \dline & 1\rline &\dots \\& & 1 & }
\end{eqnarray}
The elementary transformation $\T_i$ splits the corresponding component  to three:
\begin{eqnarray}  \label{eq:split-scheme}
\sxymatrix{\dots \rline  &1 \rline & 0 \rline \dline & 1\rline &\dots \\& & 1 & }
\end{eqnarray}
and therefore changes the number of components by 2. We call this process {\em splitting at $i$}.
The reverse process is called {\em unsplitting}.

\begin{lemma} \label{lem:vertex3}
Assume we have a graph with vertices of degrees 1,2 and 3 only and that we have no cycles. Then the moves preserve the parity of the number of components.
\end{lemma}
\begin{proof}
For vertices of degrees 1 or 2 it is clear: if the degree of vertex $j$ is 1 or 2 then $\T_j$ does not change the number of components by Lemma \ref{lem:invartiant-comp}. We check for a vertex of degree 3, which we call $i$. If we have the cases of \eqref{eq:unsplit-scheme} or \eqref{eq:split-scheme} then the number of components is changed by $\pm 2$ as all the neighboring 1's belong to different components (after the splitting). Indeed, if two of the neighboring 1's belong to the same component (after the splitting), then there exists a sequence of vertices (excluding the vertex $i$), all with label 1, connected by edges, which we write as $a_l \to v_1 \to ... \to v_m \to a_r$ where $a_l$ and $a_r$ are the vertices containing the neighboring 1's of the vertex $i$; then the graph has the following cycle: $a_l \to v_1 \to ... \to v_m \to a_r \to i \to a_l$, a contradiction to our assumption. Thus the number of components is changed by $\pm 2$ and it is clear that the parity is preserved. If we have the case
\[ \sxymatrix{\dots \rline  &1 \rline & a_i \rline \dline & 1\rline &\dots \\& & 0 & } \]
then $\T_i$ does not change it. If we have the case
\[ \sxymatrix{\dots \rline  &1 \rline & a_i \rline \dline & 0 \rline &\dots \\& & 0 & } \]
then $(\T_i(\aa))_i = 1 - a_i$ and one can check it does not change the number of components (it just expands or shrinks the component). Finally, for
\[ \sxymatrix{\dots \rline  &0 \rline & a_i \rline \dline & 0\rline &\dots \\& & 0 & } \]
clearly $\T_i$ does not change $a_i$ and thus preserve the number and parity of number of components. Due to symmetry we have covered all the cases.
\end{proof}

\begin{remark} \label{rem:loop-no-parity}
For graphs with cycle(s) the parity of number of components is not necessarily an invariant. Consider the diagram
\begin{equation}
\sxymatrix{ & a_0 \dline & \\ & a_1 \ar@{-}[ld] \ar@{-}[rd] & \\ a_2 \rline & a_3 \rline & a_4 }
\end{equation}
and denote a labeling as $\aa = (a_0;a_1;a_2,a_3,a_4)$. This diagram has a cycle ($a_1 \rightarrow a_2 \rightarrow a_3 \rightarrow a_4 \rightarrow a_1$). Assume we have the labeling $\theta_1 = (1;1;1,1,1)$. This labeling has only one component and the vertex of $a_1$ is of degree 3. We do unsplitting $\T_1$. Then we have $\theta_2 = \T_1(\theta_1) = (1;0;1,1,1)$ which has only two components.
\end{remark}


\bigskip

In the following we will consider Dynkin diagrams\cite[Table 1]{OV} and affine Dynkin diagrams\cite[Table 6]{OV}.

\section{Calculations and results}
\label{sec:Calculations}

\subsection{Diagrams of type $\AA_n$}
\label{sec:An}

Diagrams of type $\AA_n$ are the simplest case we treat. The basic properties and behavior of the moves on the labelings of Dynkin diagrams are established in this section, and thus it is advised to read this section thoroughly. Notations such as $\xi_i^n$  are defined in this section and they, along with their derivatives, will be used in the more complicated cases.

\boBE \
It is repeated here because it is a ``must read'' to understand the rest of the text.

\bigskip

The Dynkin diagram of type $\AA_n$  is
\begin{equation*}
\sxymatrix{ \bc{1} \rline & \cdots \rline & \bc{n} }
\end{equation*}


\begin{lemma} \label{lem:basic-An}
For $\AA_n$ we have:
\begin{enumerate}
\item[(a)] An elementary transformation does not change the number of components.
\item[(b)] Every component can be shrunk to length of 1, e.g. \[ (0 \ll 1 \ll 1 \ll 1 \ll 0) \mapsto (0 \ll 0 \ll 1 \ll 0 \ll0) \ . \]
\item[(c)] Components may be pushed so the space between components is of length 1, e.g. \[ (1 \ll 0 \ll 0 \ll 1 \ll 0) \mapsto (1 \ll 0 \ll 1 \ll 0 \ll 0) \ .\]
\end{enumerate}
\end{lemma}
\begin{proof} Easy. \end{proof}

\begin{definition} \label{def:xi-form}
Labelings of $\AA_n$ of the form
\begin{equation}
\xi_r^n = \left( \sxymatrix{ 1 \rline & 0 \rline & \overset{1-0}{\cdots\cdots} & \lline
\underset{r}{1} \rline & 0 \rline & \cdots } \right)
\end{equation}
which have $r$ components, defined by
\[ (\xi_r^n)_i = \begin{cases} 1 & \mbox{ for } i=1,3,\dots,2r-1 \ , \\ 0 &  \mbox{ for } i \ne
1,3,\dots,2r-1 \ ,
\end{cases}
\]
are denoted $\xi_r^n$ (or shortly $\xi_r$ when $n$ is clear from the context) and are called {\em labelings with $r$ components packed maximally to the left}.  \newline
Labelings of $\AA_n$ of the form
\[
\eta_r^n = \left( \sxymatrix{ 0 \rline &  \cdots & \lline 0 & \lline \underset{r}{1} \rline & \overset{0-1}{\cdots\cdots} &
 \lline 0 & \lline 1 } \right)
\]
which have $r$ components, defined by
\[ (\eta_r^n)_i = \begin{cases} 1 & \mbox{ for } i=n,n-2,...,n-2(r-1) \ , \\ 0 &  \mbox{ otherwise,}
\end{cases}
\]
are denoted $\eta_r^n$ (or shortly $\eta_r$ when $n$ is clear from the context) and are called {\em labelings with $r$ components packed maximally to the right}. We sometimes call $\xi_r$ and $\eta_r$ {\em the canonical representatives} as they are minimal and compact.
\end{definition}

\begin{example}
For $\AA_7$, $\xi_3^7 = \left( 1 \ll 0 \ll 1 \ll 0 \ll 1 \ll 0 \ll 0 \right)$ and $\eta_2^7 = \left( 0 \ll 0 \ll 0 \ll 0 \ll 1 \ll 0 \ll 1 \right)$.
\end{example}

As a corollary we immediately get the following fundamental theorem:
\begin{theorem} \label{th:basic-An} \label{cor:basic-An}
For a diagram of the form $\AA_n$ we have:
\begin{enumerate}
\item[(i)] The number of components is an invariant of each orbit.
\item[(ii)] Each labeling with $r$ components is equivalent to both $\xi_r^n$ and $\eta_r^n$.
\item[(iii)] Two labelings are equivalent if and only if they have the same number of components.
\end{enumerate}
\end{theorem}
\begin{proof}
(i) follows from Lemma \ref{lem:basic-An}(a). (ii) follows from Lemma \ref{lem:basic-An}(b+c) by applying shrinking of components and then pushing. In particular $\xi_r \sim \eta_r$ just by pushing all the components from one end to another. (iii) follows from (i) and (ii), in particular: if $\aa, \aa' \in L(D)$ have the same number of components $r$ then by (ii) they are both equivalent to $\eta_r$ and thus $\aa \sim \eta_r \sim \aa'$.
\end{proof}

We conclude that in $\AA_n$ the number of components is an invariant which fully characterizes classes.

\begin{corollary}\label{cor:An-reps}
\Zerois one labeling zero $\xi_0$. For \minreps we can take canonical representatives $\xi_0,\xi_1,\dots,\xi_r$, where $r=\lceil
n/2\rceil$. We have
\begin{equation} \label{eq:An-num-of-classes}
\Cls{\AA_n}=r+1=\lceil n/2\rceil+1 = \begin{cases} k+1 & \mbox{ if }n=2k \ , \\ k+2 &
\mbox{ if }n=2k+1 \ . \end{cases}
\end{equation}
\end{corollary}

\begin{remark}
The size
\begin{equation} \label{eq:max-num-of-comps-in-line}
r = \left\lceil \frac{n}{2} \right\rceil =  \begin{cases} k & \mbox{ if }n=2k \ , \\ k+1 &
\mbox{ if }n=2k+1 \ . \end{cases}
\end{equation}
is the maximal number of components we can put in a straight line diagram of $n$ vertices (e.g. $\AA_n$). This size will occur also in more complicated cases.
\end{remark}

\subsection{Diagrams of type $\tilde{\AA}_n$}
\label{sec:tilde-An}

The affine Dynkin diagram of $\tilde{\AA}_n$, denoted $\AA_n^{(1)}$ in \cite[Table 6]{OV}, is the following diagram with $n+1$ vertices ($n\ge2$) numbered $\{0,1,...,n\}$.
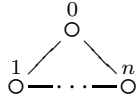
\begin{equation*}
\sxymatrix{ & \bc{0} & \\ \bc{1} \rline \ar@{-}[ru] & \cdots \rline & \ar@{-}[lu] \bc{n} }
\end{equation*}
We write a labeling $\aa = (a_i)_{i=0}^n$ as $(a_0;a_1,...,a_n)$.

This diagram is not acyclic (in fact, it is a cycle) but as each vertex is of degree 2 it is simple enough to be treated directly based on section \ref{sec:basics}.

\begin{lemma} \label{lem:tilde-An-fixed-points1}
The labelings $\ell_0 = (0;0,...,0)$  and $\ell_1 = (1;1,...,1)$ are fixed labelings.
\end{lemma}
\begin{proof}
For $(0;0,...,0)$ it is clear. For $(1;1,...,1)$ consider $\T_i$ for any $i$. Then if $\aa' = \T_i(\aa)$ then clearly $a'_j = a_j$ and
\[ a'_i = a_i + a_{i-1} + a_{i+1} = 1 + 1 + 1 = 1 = a_i \]
and thus $\aa' = \aa$.
\end{proof}

\begin{lemma} \label{lem:tilde-An-fixed-points2}
If $n=2k+1$ is odd then the labelings $(0;\xi_{k+1}) = (0;1,0,\overset{10}{...},1)$  and $\ell_2 = (1;0,1,\overset{01}{...},0)$ are fixed labelings.
\end{lemma}
\begin{proof}
Note that each vertex has an even number of neighbors with labels 1, and by Remark \ref{rem:elem-move} the lemma follows.
\end{proof}

\begin{lemma} \label{lem:tilde-An-components}
The moves do not change the number of components. In particular, excluding the fixed labelings, the number of components completely characterizes an orbit.
\end{lemma}
\begin{proof}
As each vertex has degree 2 it follows from Lemma \ref{lem:invartiant-comp}.
\end{proof}

\begin{proposition} \label{prop:tilde-An}
The number of classes is
\[
\Cls{\tilde{\AA}_n} = \left\lceil \frac{n}{2} \right\rceil + 2 + (n \bmod 2) =
 \begin{cases} k+2 & \mbox{ if } n=2k \ , \\ k+4 & \mbox{ if } n=2k+1 \ . \end{cases}
\]
\end{proposition}

\begin{proof}[Intuitive proof]
To calculate the number of classes we need to count how many components we can insert in the diagram, and also add the number of the fixed labelings. The term $\lceil n/2 \rceil$ is the maximal number of components we can insert in $n$ vertices, provided that $a_1=1$. The new vertex 0 cannot accept new component as it is a neighbor of $a_1$. The term 2 is for the fixed labelings of Lemma \ref{lem:tilde-An-fixed-points1}. By Lemma \ref{lem:tilde-An-fixed-points2} when $n$ is odd we need to count also $\ell_2$ (the fixed labeling $(0;\xi_{k+1})$ is counted by the first term), and this is accounted by the term $(n \bmod 2)$.
\end{proof}

And now we give the formal proof:

\begin{proof}
Let us consider the map $\varphi : L(\AA_n) \to L(\tilde{\AA}_n)$ defined by
\[ \varphi(a_1,...,a_n) = (0;a_1,...,a_n) \ . \]
Consider the induced map $\varphi_* : \Cl(\AA_n) \to \Cl(\tilde{\AA}_n)$ defined by taking the quotient, that is:
\[ \varphi_*([(a_1,...,a_n)]) = [\varphi(a_1,...,a_n)] = [(0;a_1,...,a_n)] \ . \]
Then $\varphi_*$ is well-defined. If $\aa' \sim \aa$ in $\AA_n$ then clearly $\varphi(\aa') \sim \varphi(\aa)$ (use the same $\T_i$ for both $\aa$ and $\varphi(\aa)$).


The map $\varphi_*$ is also injective. Suppose $\aa' \sim \aa \in \Img \varphi$. Then we can assume $\aa, \aa' \notin \{ \ell_1, \ell_2 \}$ and that $\aa' = (0;a'_1,...,a'_n) \sim (0;a_1,...,a_n) = \aa$. By Lemma \ref{lem:tilde-An-components} $\aa = (0;a_1,...,a_n)$ and $\aa' = (0;a'_1,...,a'_n)$ have the same number of components. Then, as $\aa = \varphi(a_1,...,a_n)$ and $\aa' = \varphi(a'_1,...,a'_n)$, we have that $(a_1,...,a_n)$ and $(a'_1,...,a'_n)$ have the same number of component (because in both $a_0=a'_0=0$), and thus by Theorem \ref{th:basic-An} we get $(a_1,...,a_n) \sim (a'_1,...,a'_n)$. Thus, $[\aa] = \varphi_*([a_1,...,a_n]) = \varphi_*([(a'_1,...,a'_n)]) = [\aa']$ implies $[(a_1,...,a_n)]=[(a'_1,...,a'_n)]$.

\medskip

We now calculate the image of $\varphi_*$. And for that we need the following sublemma:

\textbf{SubLemma:} Let $\aa = (a_0; a_1,...,a_n)$ be a labeling. If $a_0 = 1$ and $\aa \ne \ell_1 , \ell_2$ then $\aa$ is equivalent to a labeling $\aa' = (a'_0;a'_1,...,a'_n)$ with $a'_0 = 0$.

{\em Proof:} We show that if a labeling is not a fixed labeling then we can change the label $a_0$ of the zeroth vertex to 0 by a sequence of moves. Assume $a_0=1$ (note that we have left-right symmetry). If we can do this, we say we can {\em win} and after we managed to change $a_0$ to $a'_0=0$ we say we {\em won}.
\begin{itemize}
\item If $a_1=1$ and $a_n=0$ just apply $\T_0$ and we won. If $a_1=0$ and $a_n=1$ we won again (symmetric cases).
\item Now assume $a_1=a_n=0$. We can change $a_0$ to 0 by moving the component of $a_0$. If $a_2=0$ we won, apply $\T_1$ and then $\T_0$. If $a_2=1$, $a_3=1$ and $a_4=0$ we won again (apply $\T_3$ and $\T_2$ and then $\T_1$ and $\T_0$). So in order not to win we need that the component cannot be moved. This happens when:
    \begin{itemize}
    \item[(i)] Either $a_i=1$ for all $i=1,...,n$ which is excluded by the assumption that $\aa \ne \ell_1$ and also by the case assumption $a_1=a_n=0$.
    \item[(ii)] Either $n=2k$ and we have $a_2=a_4=...=a_n=1$, $a_3=a_5=...=a_{n-1}=0$ (the alternating labeling) which is a contradiction to $a_n=0$. In other words, the case $\ell_2$ cannot occur here due to parity considerations. Either $n=2k+1$ and we have $a_2=a_4=...=a_n=1$, $a_3=a_5=...=a_{n-1}=0$ (the alternating labeling). But this ``problematic'' labeling is $\ell_2$ which is excluded by our assumption.
    \end{itemize}
\item Now assume $a_1=a_n=1$. We can change $a_0$ to 0 by shrinking the component from one side. If $a_2=0$ we won by applying $\T_1$ and then $\T_0$. If $a_2=1$ and $a_3=0$ we can win again ($\T_2$,$\T_1$,$\T_0$). So in order not to win we need that the $a_0=1$ component cannot be shrunk from any side. This happens only when
    $a_1=a_2=a_3=...=a_{n-1}=a_n=1$ but since $a_0=1$ by assumption we get that $\aa = (1;1,...,1) = \ell_1$ which is a contradiction to our assumption.
\end{itemize}
Thus we can always win under our assumptions, and the sublemma is proved.

\bigskip

As a corollary we get that if $\aa \ne \ell_1, \ell_2$ then $[\aa] \in \Img \varphi_*$ since by the sublemma, if $\aa = (a_0,a_1,...,a_n) \ne \ell_1, \ell_2$ then we can change it by a sequence of moves such that $a'_0=0$ and this implies that $\aa \sim \aa' = \varphi(\aa'')$ where $\aa' = (0;\aa'')$, and thus $[\aa] = \varphi_*([\aa''])$ and $[\aa] \in \Img \varphi_*$. Note also that $[\ell_1] = \{ \ell_1 \}$ and $[\ell_2] = \{ \ell_2 \}$ are the only classes not in the image of $\varphi_*$. Thus, if $n=2k$ then $\Img \varphi_* = \Cl(\tilde{\AA}_n ) \smallsetminus \{ [\ell_1] \}$ and if $n=2k+1$ then $\Img \varphi_*  = \Cl(\tilde{\AA}_n ) \smallsetminus  \{ [\ell_1] , [\ell_2] \}$.
\medskip

Thus we conclude, if $n=2k$ then
\[
\Cl(\tilde{\AA}_n ) = \Img \varphi_* \sqcup \{ [\ell_1] \} \cong \Cl(\AA_n ) \sqcup \{ [\ell_1] \}
\]
which implies
\[
 \Cls{\tilde{\AA}_n} = \Cls{\AA_n} + 1 = \left\lceil \frac{n}{2} \right\rceil + 1 + 1 = \left\lceil \frac{n}{2} \right\rceil + 2 = k + 2 \, \]
and if $n=2k+1$ then
\[
\Cl(\tilde{\AA}_n ) = \Img \varphi_* \sqcup \{ [\ell_1] , [\ell_2] \} \cong \Cl(\AA_n )  \sqcup \{ [\ell_1] , [\ell_2] \}
\]
which implies
\[
\Cls{\tilde{\AA}_n} = \Cls{\AA_n} + 2 = \left\lceil \frac{n}{2} \right\rceil + 1 + 2 = \left\lceil \frac{n}{2} \right\rceil + 3 = k+1 + 3 = k+4 \ ,
\]
and we are done.
\end{proof}

\begin{corollary}
For $\tilde{\AA}_n$ we can take the following \minreps: $\ell_1 = (1;1,...,1), \ (0;\xi_0), \ (0;\xi_1), \ ... \ , (0;\xi_r)$ where $r=\lceil n/2 \rceil$ and $\ell_2$ when $n=2k+1$.
\end{corollary}
\begin{proof}
By Proposition \ref{prop:tilde-An} and Lemma \ref{lem:tilde-An-components} we see that by excluding the fixed labelings, the number of components completely characterizes the class. We add the fixed labelings to account for the classes with 1 component of all 1's, and $k+1$-components $\ell_2$ when $n=2k+1$.
\end{proof}

\subsection{Diagrams of type $\BB_n$}
\label{sec:Bn} \label{subsec:Bn}

\expBE

The diagram $\BB_n$  is
\begin{equation*}
\sxymatrix{ \bc{1} \rline & \cdots & \lline \bc{n-1} \ar@{=>}[r] & \bc{n} }
\end{equation*}
(this Dynkin diagram corresponds to the compact group $\Spin(2n+1)$ which is the universal cover of the special orthogonal group $\SO(2n+1)$).
The long vertex does not see the short vertex: i.e. they are not affected by the label at the short
vertex. Writing every labeling $\bb$ as
\[ \bb = (\a \RRR \vk), \]
where $\vk \in \{0,1\}$ is the label at the short vertex and $\aa \in L(\AA_{n-1})$ is the tail with $n-1$ vertices,
the moves $\T^B_i$ of $L(\BB_n)$ act as follows:
\begin{enumerate}
\item[(a)] For $i < n-1$ : $\T^B_i(\bb) = (\T^A_i(\a) \Rightarrow \vk)$.
\item[(b)] For $i = n-1$ : $\T^B_{n-1}(\bb) = (\T^A_{n-1}(\a) \Rightarrow \vk)$.
\item[(c)] For $i = n$ : $(\T_n^B(\bb) )_n = \vk + (\bb)_{n-1}$.
\end{enumerate}
(here $\T^A_i$ is the move on $L(\AA_{n-1})$).

\begin{lemma}[The Swallowing Lemma] \label{lem:Bn-shortroot-shrink}
 If $\a\in L(\AA_{n-1})$, $\a\neq 0$, then $\a \RRR 0$ is equivalent to $\a\RRR 1$.
\end{lemma}

\begin{proof}
Consider the rightmost component of 1's in the long roots (it clearly exists since $\aa \ne 0$). It can be expanded to the right (see
Lemma \ref{lem:basic-An}), {\em swallow} the short root's 1 (by move $\T_n$)
and then shrink back to the original place and length,
\[
\left( \sxymatrix{ \cdots \rline & 1 \rline & 0  \ar@{=>}[r] & 1 } \right)
\overset{\T_{n-1}}{\longmapsto} \left( \sxymatrix{ \cdots \rline & 1 \rline & 1  \ar@{=>}[r] & 1
}\right)
 \overset{\T_{n}}{\longmapsto} \left( \sxymatrix{ \cdots \rline & 1 \rline & 1  \ar@{=>}[r] & 0 } \right)
  \overset{\T_{n-1}}{\longmapsto} \left( \sxymatrix{ \cdots \rline & 1 \rline & 0  \ar@{=>}[r] & 0 }  \right) \ .
\]
The actions of $\T_{n-1}$ and $\T_n$ described in (a)--(c) above assure us that this algorithm works: property (a) allows us to do the usual manipulations
on the tail $\AA_{n-1} = (b_i)_{i=1}^{n-1}$, property (b) implies that $\T_{n-1}$ can change the label $b_{n-1}$ from 0 to 1 even when $\vk=1$,
then (c) gives us swallowing by $\T_n$ and applying $\T_{n-1}$ by (b) give us shrinking back.
\end{proof}

\begin{remark}
The labeling
\begin{equation*}
\ell_1^{(0)} = \left( \xi_0 \RRR 1 \right) = \left( 0 \ll ... \ll 0 \RRR 1 \right) \ .
\end{equation*}
is a fixed labeling. We denote by $[\ell_1^{(0)}] \in \Or(\BB_n)$ the orbit $\{\ell_1^{(0)}\}$ of $\ell_1^{(0)}$ and by slight abuse of notation also the set $\{ [\ell_1^{(0)}] \}$.
\end{remark}

\begin{lemma} \label{lem:Bn-bijection}
The map
\begin{equation}
\varphi :
\begin{array}{rcl}
L(\AA_{n-1}) & \longrightarrow & L(\BB_n) \\ \aa & \longmapsto & (\sxymatrix{\a \arr &0}) \end{array}
\end{equation}
induces a bijection on the orbits $\varphi_* \colon \Orbset(\AA_{n-1}) \isoto \Orbset(\BB_n) \smallsetminus \{[\ell_1]\}$ defined by $\varphi_*([\a]) = [ \varphi(\a)]$.
\end{lemma}

\begin{proof}

We now need to show that $\a \Wsim{A} \a' \iff \varphi(\a) \Wsim{B} \varphi(\a')$. 
That $\a \Wsim{A} \a' \implies \varphi(\a) \Wsim{B} \varphi(\a')$ is clear. Conversely, if $\varphi(\a) \Wsim{B} \varphi(\a')$. Then $(\a \RRR 0) \Wsim{B} (\a' \RRR 0)$ which implies $\aa$ and $\aa'$ has the same number of components, and by Theorem \ref{th:basic-An} $\a \Wsim{A} \a'$.
Thus $\varphi_*([\a]) = \varphi_*([\a']) \implies [\a]=[\a']$ and thus $\varphi_*$ is injective. Thus the map is well-defined and injective.

Let $\bb \in L(\BB_n) \smallsetminus \{ \ell_1^{(0)} \}$. If $\bb \ne 0_B$ then $\bb$ has at least one
long root with label 1. Then by Lemma \ref{lem:Bn-shortroot-shrink} we have $\bb \Wsim{B}
(\aa \RRR 0) = \varphi(\aa)$ for some $\a \in L(\AA_{n-1})$. Thus $[\bb] = \varphi_*([\a])$.
If $\bb = 0_B$ then clearly $\bb = \varphi(0_A)$. This implies $\varphi_*$ is surjective on
$\Orbset(\BB_n)\smallsetminus \{[\ell_1^{(0)}]\}$. Clearly, $\ell_1^{(0)} \ne \varphi(\a)$ for any $\a \in
L(\AA_{n-1})$ since $(\ell_1)_n = 1$. Since $\ell_1$ is a fixed labeling then $\ell_1^{(0)} \not\Wsim{B} \varphi(\a)$ for any $\a \in L(\AA_{n-1})$.
These show that $\im \varphi_* = \Orbset(\BB_n) \smallsetminus \{ [\ell_1^{(0)}] \}$.


We conclude that $\varphi_* \colon \Orbset(\AA_{n-1}) \isoto \Orbset(\BB_n) \smallsetminus \{[\ell_1^{(0)}]\}$ is bijective and the lemma follows.
\end{proof}


\begin{corollary} \label{cor:Bn-reps} For $\BB_n$ we have:
\begin{enumerate}
\item[(i)]
\Zerois the invariant 0 and represented by $(\xi_0\! \RRR \! 0)$.
\item[(ii)]
For \reps we can take
$(\xi_0\! \RRR \!1), \ (\xi_0 \! \RRR \! 0), \ (\xi_1\! \RRR \! 0), \ (\xi_2 \! \RRR \! 0) \ , \ ... \ , (\xi_r \! \RRR \! 0)$
with $r = k =
\lceil \frac{n-1}{2} \rceil$. Note that by Lemma \ref{lem:Bn-shortroot-shrink} we can choose the
short root label to be 0 whenever $j > 0$.
\item[(iii)]
The number of orbits is
\begin{eqnarray}
\# H^1(\R,G)=\Orbs{\BB_n} & = & \Orbs{ \AA_{n-1} } + 1 = 2 + \left\lceil \frac{n-1}{2} \right\rceil \nonumber
\\ & = & \begin{cases} k+2 & \mbox{ if } n=2k \ , \\ k+2 & \mbox{ if }
n=2k+1 \ . \end{cases}
\end{eqnarray}
\end{enumerate}
\end{corollary}
\begin{proof}
(i) is easy. (ii) follows from the bijectivity of the map constructed in Lemma \ref{lem:Bn-bijection}. (iii) follows from (ii).
\end{proof}

\subsection{Diagrams of type $\tilde{\BB}_n$} \label{subsec:tilde-Bn}
\label{sec:tilde-Bn}

The affine Dynkin diagram of type $\tilde{\BB}_n$ with $n+1$ vertices labeled $\{0,1,...,n\}$ with $n \ge 4$ is
\begin{equation*}
\sxymatrix{ \bc{1} \rline & \bc{2} \dline \rline & \cdots \rline & \bc{n-1} \arr & \bc{n} \\
& \bcu{0} & & }
\end{equation*}
It is denoted $\BB_n^{(1)}$ in \cite[Table 6]{OV}.
We introduce short notation for this diagram:
\begin{equation*}
\bb = \left( \frac{b_1}{b_0}b_2 \cdots b_{n-1} \RRR b_n \right)
\end{equation*}
We call vertices 0 and 1 {\em leafs}, vertices $2,...,n-1$ {\em the neck} and to the short vertex $n$ we call {\em ghost}.

\begin{remark} \label{lem:tilde-Bn-fixed-points}\label{rem:tilde-Bn-fixed-points}
For $\tilde{\BB}_n$ we always have the following fixed labelings:
\begin{equation*}
\ell_0 = \frac{0}{0}0...0\RRR0 , \quad \ell_2 = \frac{1}{1}0...0\RRR0 , \quad \ell_1 = \frac{0}{0}0...0\RRR1 , \quad \ell_3 = \frac{1}{1}0...0\RRR1 .
\end{equation*}
\end{remark}

\begin{proposition}
If $n=2k+1$ then $\Cls{\tilde{\BB}_n} = k+4$.
\end{proposition}
\begin{proof}
We construct a map $\varphi : L(\AA_{n-2}) \to L(\tilde{\BB}_n)$ defined by
\[ \aa \longmapsto \bb = \frac{0}{0}\aa \RRR 0 \]
and the quotient map $\varphi_* \Cl(\AA_{n-2}) \to \Cl(\tilde{\BB}_n)$ defined by $\varphi_*([\aa]) = [\varphi(\aa)]$. Clearly, the quotient map is well-defined. It is also injective. Let $\bb' \sim \bb \in \Img \varphi$, then $\bb = \frac{0}{0}\aa \RRR 0$ and $\bb' = \frac{0}{0}\aa' \RRR 0$, and clearly they both have the same number of components, which implies that $\aa$ and $\aa'$ have the same number of components. Theorem \ref{th:basic-An} implies that $\aa \sim \aa'$. We showed that if $\varphi_*([\aa]) = [\bb] = [\bb'] = \varphi_*([\aa'])$ then $[\aa] = [\aa']$.

We now need to compute its image and show that
\[ \Img \varphi_* = \Cl(\tilde{\BB}_n) \smallsetminus \{ [\ell_1],[\ell_2],[\ell_3] \} . \]
Clearly, those fixed labelings are not in the image of $\varphi_*$. It remains to show that all the other classes are in the image of $\varphi_*$. So we need to show two things:
\begin{enumerate}
\item If $b_n=1$ and $\bb$ is not a fixed labeling, it is equivalent to a labeling with $b'_n=0$.
\item Every non-fixed labeling is equivalent to a labeling with $b'_0=b'_1=0$.
\end{enumerate}

Assertion 1 follows from the Swallowing Lemma (see Section \ref{sec:Bn}) as we need at least 1 component in the neck. If the neck has no components then: if $b_0=b_1=0$ then $\bb = \ell_1$ which we assumed is not the case, if $b_0=b_1=1$ then $\bb = \ell_3$ which we assumed is not the case, if $b_0=1, b_1=0$ we can push the component to the neck and the swallow the ghost's 1.

Assertion 2 is proved for the cases $\DD_n$ and $\tilde{\DD}_n$ (see Sections \ref{subsec:Dn} and \ref{subsec:tilde-Dn}), and the same proofs also work here.

Finally we note that the maximal number of components we can insert in the neck is $k$. So we have
\[
\Cls{\tilde{\BB}_n} = \Cls{\AA_{n-2}} + 3 = (\left\lceil \frac{n-2}{2} \right\rceil + 1) + 3 = \left\lceil \frac{2k-1}{2} \right\rceil + 1 + 3 = k+4 \ .
\]
\end{proof}

\begin{remark} \label{lem:tilde-Bn-semifixed}\label{rem:tilde-Bn-semifixed}
For $n=2k$ we have the following two classes:
\begin{eqnarray}
\kappa_0 & = & \left\{ \frac{0}{1}\eta_{k-1}^{n-2}\RRR0 , \frac{0}{1}\eta_{k-1}^{n-2}\RRR1 \right\} \\
\kappa_1 & = & \left\{ \frac{1}{0}\eta_{k-1}^{n-2}\RRR0 , \frac{1}{0}\eta_{k-1}^{n-2}\RRR1 \right\}
\end{eqnarray}
We call $\kappa_0$ and $\kappa_1$ the {\em semi-fixed classes}.
\end{remark}

For example:
\[ \tilde{\BB}_4 , \quad \kappa_1 = \left\{ \frac{1}{0}01\RRR0 , \ \frac{1}{0}01\RRR1 \right\} \ . \]

\begin{proposition}
If $n=2k$ then $\Cls{\tilde{\BB}_n} = k+5$.
\end{proposition}
\begin{proof}
We construct a map $\varphi : L(\AA_{n-2}) \to L(\tilde{\BB}_n)$ defined by
\[ \aa \longmapsto \bb = \frac{0}{0}\aa \RRR 0 \]
and the quotient map $\varphi_* \Cl(\AA_{n-2}) \to \Cl(\tilde{\BB}_n)$ defined by $\varphi_*([\aa]) = [\varphi(\aa)]$. Clearly, the quotient map is well-defined. It is also injective. Let $\bb' \sim \bb \in \Img \varphi$, then $\bb = \frac{0}{0}\aa \RRR 0$ and $\bb' = \frac{0}{0}\aa' \RRR 0$, and clearly they both have the same number of components, which implies that $\aa$ and $\aa'$ have the same number of components. Theorem \ref{th:basic-An} implies that $\aa \sim \aa'$. We showed that if $\varphi_*([\aa]) = [\bb] = [\bb'] = \varphi_*([\aa'])$ then $[\aa] = [\aa']$.

We now need to compute its image and show that
\[ \Img \varphi_* = \Cl(\tilde{\BB}_n) \smallsetminus \{ [\ell_1],[\ell_2],[\ell_3], \kappa_0 , \kappa_1 \} . \]
Clearly, those fixed labelings are not in the image of $\varphi_*$ and also $\kappa_0, \kappa_1 \notin \Img \varphi_*$. It remains to show that all the other classes are in the image of $\varphi_*$. So we need to show two things:
\begin{enumerate}
\item If $b_n=1$ and $\bb$ is not a fixed labeling, it is equivalent to a labeling $\bb'$ with $b'_n=0$.
\item Every non-fixed labeling $\bb$ which is also not in a semi-fixed orbit, is equivalent to a labeling $\bb'$ with $b'_0=b'_1=0$.
\end{enumerate}

Assertion 1 follows from the Swallowing Lemma (see Section \ref{sec:Bn}) as we need at least 1 component in the neck. If the neck has no components then: if $b_0=b_1=0$ then $\bb = \ell_1$ which we assumed is not the case, if $b_0=b_1=1$ then $\bb = \ell_3$ which we assumed is not the case, if $b_0=0, b_1=1$ we can push the component to the neck and the swallow the ghost's 1.

For assertion 2. If $b_0=b_1=1$ and we have a component in the neck we can do unsplitting and swallow the the 1's in the leafs:
\[ \frac{1}{1}01... \longmapsto \frac{1}{1}11... \longmapsto \frac{0}{0}11... \]
(the case where $b_2=1$ also is trivial). If $b_0=1$ and $b_1=0$ (the other case $b_0=0$ and $b_1=1$ is symmetric) and if $b_2=0$ (if $b_2=1$ the case is trivial) then in order to make the leafs all 0's we need to move the 1 in $b_0$ to the neck. One can see that we cannot do it only if there are $k-1$ components in the neck (the neck is $\eta_{k-1}^{n-2}$), and then $\bb \in \kappa_0$ ($\bb \in \kappa_1$ for the other symmetric case) which we assumed is not the case.

Finally we note that the maximal number of components we can insert in the neck is $k-1$. So we have
\begin{eqnarray*}
\Cls{\tilde{\BB}_n} & = & \Cls{\AA_{n-2}} + 5 = (\left\lceil \frac{n-2}{2} \right\rceil + 1) + 5 \\
 & = & (\left\lceil \frac{2k-2}{2} \right\rceil + 1) + 5 = k-1 + 1 + 5 = k+5 \ .
\end{eqnarray*}
\end{proof}

\begin{corollary}
For $\tilde{\BB}_n$ we have
\begin{equation*}
\Cls{\tilde{\BB}_n} = \begin{cases} k+4 & \mbox{ if  } n=2k+1, \\ k+5 & \mbox{ if  } n=2k . \end{cases}
\end{equation*}
For the odd case $n=2k+1$ the \minreps are
\[ \ell_0, \ell_1, \ell_2, \ell_3, \quad \mbox{ and } \quad \left( \frac{0}{0}\xi_i\RRR0 \right) \mbox{ for } \ i=1,...,k \ . \]
For the even case $n=2k$ the \minreps are
\[ \ell_0, \ell_1, \ell_2, \ell_3, \quad \mbox{ and } \quad \left( \frac{0}{0}\xi_i\RRR0 \right) \mbox{ for } i=1,...,k-1 \quad \mbox{ and } \quad \left( \frac{0}{1}\eta_{k-1}^{n-2}\RRR0 \right) , \ \left( \frac{1}{0}\eta_{k-1}^{n-2}\RRR0 \right) . \]
\end{corollary}

\begin{example}
For
\begin{equation*}
\tilde{\BB}_4 = \sxymatrix{ \bc{1} \rline & \bc{2} \dline \rline & \bc{3} \arr & \bc{4} \\
& \bcu{0} &  }
\end{equation*}
we have $n=4, k=2$ and the 7 = 2 + 5 representatives are
\[
\ell_0 = \left( \frac{0}{0}00\RRR0 \right) , \ \left( \ell_1 = \frac{0}{0}00\RRR1 \right) , \ \ell_2 = \left( \frac{1}{1}00\RRR0 \right) , \ \left( \ell_3 = \frac{1}{1}00\RRR1 \right) , \
\]
and
\[ \frac{0}{0}10\RRR0 \ , \quad \frac{0}{1}01\RRR0 \quad \mbox{ and } \quad \frac{1}{0}01\RRR0 \ . \]
\end{example}

\begin{example}
For
\begin{equation*}
\tilde{\BB}_5 = \sxymatrix{ \bc{1} \rline & \bc{2} \dline \rline & \bc{3} \rline & \bc{4} \arr & \bc{5} \\
& \bcu{0} & & }
\end{equation*}
we have $n=5, k=2$ and the 6 = 2 + 4 representatives are
\[
\ell_0 = \frac{0}{0}000\RRR0 , \ \ell_1 = \frac{0}{0}000\RRR1 , \ \ell_2 = \frac{1}{1}000\RRR0 , \ \ell_3 = \frac{1}{1}000\RRR1 , \
\]
and
\[ \frac{0}{0}100\RRR0 \quad \mbox{ and } \quad \frac{0}{0}101\RRR0 \ . \]
\end{example}


\subsection{The diagram $\CC_n$} \label{subsec:Cn}
\expBE

The Dynkin diagram of of type $\CC_n$ with $n \ge 3$ is
\begin{equation*}
\sxymatrix{ \bc{1} \rline & \cdots & \lline \bc{n-1} & \all \bc{n} }
\end{equation*}
where the most right vertex is the longer vertex. The long vertex does not see the short vertices and
therefore is not affected by them (see formula \eqref{eq:move-nonsimple} in section \ref{sec:rules}).
(This Dynking diagram corresponds to the compact ``quaternionic'' group $\Sp_n=\Sp(n)$ with $n \ge 3$.)

We state formally how moves act on $\CC_n$ and what we mean by ``the long vertex does not see the short vertices''.
\begin{enumerate}
\item[(a)] $\T_i^C \left( \a \LLL \vk \right) = \left( \T_i^A(\a) \LLL \vk \right)$ for $i < n-1$.
\item[(b)] $\left( \T_{n-1}^C(\a \LLL \vk) \right)_n = a_{n-2} + a_{n-1} + \vk$ for $i = n-1$.
\item[(c)] $\T_n^C \left( \a \LLL \vk \right) = \left( \a \LLL \T_n(\vk) \right) = \left( \a \LLL \vk \right)$ for $i=n$.
\end{enumerate}
(b) and (c) are derived from \eqref{eq:move-nonsimple} since $\alpha_{n-1}$ is a shorter vertex and $\alpha_n$ (with label $\vk$) is a longer vertex. Thus, the longer vertex's label cannot be changed by any move.

In order to copmute the classes of $\CC_n$ and later for $\tilde{\CC}_n$ we need to define two auxiliary diagrams.
\begin{construction}
We define
\begin{equation}
\AA_m^{(m)} = \sxymatrix{ \bc{1} \rline & \cdots & \lline \bc{m} \rline & \boxone }
\end{equation}
where the {\em boxed 1} means that for the additional $m+1$-th vertex the label 1 cannot be changed by moves (that is, we only have the moves $\T_1,...,\T_m$ in our arsenal). One can easily see (by property (c) above) that the boxed one neighboring vertex $m$ corresponds to the longer vertex $m+1$. Similarly one defines  \begin{equation}
\AA_m^{(1,m)} = \sxymatrix{ \boxone \rline & \bc{1} \rline & \cdots & \lline \bc{m} \rline & \boxone }
\end{equation}
We use the notations $a_1 \cdots a_m \ll \boe$ and $\boe\ll a_1 \cdots a_m \ll \boe$ for the labelings.
\end{construction}

We now calculate the classes of the two auxiliary diagrams.

\begin{lemma} \label{lem:aux-for-Cn}
For $\AA_m^{(m)}$ and $\AA_m^{(1,m)}$ the number of components (including the boxed 1) is invariant under elementary transformation and is unique to each orbit. For $\AA_m^{(1,m)}$ the labeling $\boe\ll1...1\ll\boe$ (all 1's) is a fixed labeling and it is the only labeling with one component.
\end{lemma}
\begin{proof}
Follows from Lemma \ref{lem:basic-An} and Theorem \ref{th:basic-An} as in the case of $\AA_n$.
\end{proof}

It is easily seen that $\AA_m^{(m)}$ and $\AA_m^{(1)}$ are essentially the same diagram and thus $\Cl(\AA_m^{(m)}) \cong \Cl(\AA_m^{(1)})$.

Lemma \ref{lem:aux-for-Cn} implies:
\begin{proposition} \label{prop:aux-Cn}
For $\AA_m^{(m)}$ the number of classes is
\begin{equation}
\Cls{\AA_m^{(m)}} = \left\lceil \frac{m-1}{2} \right\rceil + 1 = \begin{cases} k+1 & \mbox{ if  } m = 2k+1, \\  k+1 & \mbox{ if  } m = 2k. \end{cases}
\end{equation}
and we can take as minimal representatives for the classes the labelings $\xi_i^{m-1}0\ll\boe$ for $i=0,...,k$ where $k = \left\lceil \frac{m-1}{2} \right\rceil$.
For $\AA_m^{(1,m)}$ the number of classes is
\begin{equation}
\Cls{\AA_m^{(1,m)}} = \left\lceil \frac{m-2}{2} \right\rceil + 2 = \begin{cases} k+1 & \mbox{ if  } m = 2k, \\  k+2 & \mbox{ if  } m = 2k+1. \end{cases}
\end{equation}
and we can take as minimal representatives for the classes the labelings $\boe\ll1..1\ll\boe$ (all 1's) and $\boe\ll0\xi_i^{m-2}0\ll\boe$ for $i=0,...,k$ where $k = \left\lceil \frac{m-2}{2} \right\rceil$.
\end{proposition}

\begin{proof}[Idea of the proof]
We see that the number of classes is determined by the maximal number of components we can insert in the ``free'' vertices, which are the vertices without a boxed 1 as a neighbor. As we count the components of the boxed 1's it is clear that putting 1 in the $m$-th vertex for $\AA_m^{(m)}$ won't add a new component since it is already adjunct to a 1. For $\AA_m^{(m)}$ we have $m-1$ free vertices and for $\AA_m^{(1,m)}$ we have $m-2$ free vertices. Calculating how many components we can insert is done as in Section \ref{sec:An}.
\end{proof}

\begin{remark}
The outline of the formal proof is to construct the injections $\aa \mapsto \aa\ll0\ll\boe$ and $\aa \mapsto \boe\ll0\ll\aa\ll0\ll\boe$, take their quotient and calculate their images, as we done in previous sections.

Note that Proposition \ref{prop:aux-Cn} is a special case of $\AA_n^{(m)}$ and $\AA_n^{(m,n)}$ with $m=1$, treated in \cite[Sections 7.2 and 9.3]{BE-real}.
\end{remark}

Now we can prove:

\begin{lemma} \label{lem:Cn}
We have a map $\varphi \colon L(\CC_n) \to L(\AA_{n-1}) \sqcup L(\AA_{n-1}^{(n-1)})$ such that the induced map
$$\varphi_* \colon \Cl(\CC_n)  \isoto \Cl(\AA_{n-1}) \sqcup \Cl(\AA_{n-1}^{(n-1)})$$ is a bijection.
\end{lemma}

\begin{proof}
%
%

We construct a bijection $\varphi : L(\CC_n) \longrightarrow L(\AA_{n-1}) \sqcup
L(\AA_{n-1}^{(n-1)})$ as follows: let $\cc = (\a \LLL \vk) \in L(\CC_n)$, if $\vk=0$ we send
$\cc$ to $\a \in L(\AA_{n-1})$ and if $\vk=1$ we send $\cc$ to $(\a \ll \boe ) \in L(\AA_{n-1}^{(n-1)})$.
The inverse map would be sending $\a \in L(\AA_{n-1})$ to $(\a \LLL 0) \in L(\CC_n)$ and
$(\a\ll\boe) \in L(\AA_{n-1}^{(n-1)})$ to $(\a \LLL 1)  \in L(\CC_n)$. We see that both maps are
well-defined.
We obtain an induced map $\varphi_* : \Cl(\CC_n)  \longrightarrow \Cl(\AA_{n-1}) \sqcup
\Cl(\AA_{n-1}^{(n-1)})$ defined by $\varphi_*([\cc]) = [ \varphi(\cc) ]$. We check that the map $\varphi_*$ is well-defined and bijective, and the lemma will follow from the bijectivity of $\varphi_*$.

Clearly we have $\cc \Wsim{C} \cc' \iff \varphi(\cc) \Wsim{A} \varphi(\cc')$.
Indeed $$\cc \sim \cc' \iff (\a \LLL \vk) \sim (\a' \LLL \vk)$$ (as the
label of $\vk$ cannot be changed by elementary transformation $\T_n$ as explained above). If $\vk=0$
we have $$(\a \LLL 0) \sim (\a' \LLL 0) \iff \a \Wsim{A} \a'$$ and if $\vk=1$
we have $$(\a \LLL 1) \sim (\a' \LLL 1) \iff (\a \ll\boe) \sim (\a' \ll\boe) \ . $$

The map $\varphi_*$ is injective: if $\varphi(\cc) \sim \varphi(\cc')$ then either $\a \sim \a'$ and then $\cc \sim \cc'$
for $c_n=0$ or $(\a \ll \boe) \sim (\a' \ll \boe)$ and $\cc \sim \cc'$ for $c_n=1$.

This map is clearly surjective: for $[\a] \in \Orbset(\AA_{n-1})$ we have $\varphi_*([\a \LLL 0]) = [\a]$
and for $[(\a \ll\boe)] \in \Cl(\AA_{n-1}^{(n-1)})$ we have $\varphi_*([\a \LLL 1]) = [(\a \ll \boe]$.
Thus, $\varphi_*$ is a bijection and the lemma follows.
\end{proof}


\begin{corollary} \label{cor:Cn}
For $\CC_n$ \minreps are:
\[ (\xi_0 \LLL 0) \ , \quad (\xi_1 \LLL 0) \ , \quad  \cdots \ , \quad (\xi_r \LLL 0) \]
where $r = \lceil \frac{n-1}{2} \rceil$,
and
\[ (\xi_0 \LLL 1) \ , \quad (\xi_1 \LLL 1) \ , \quad  \cdots \ , \quad (\xi_s \LLL 1) \]
where $s = \lceil \frac{n}{2} \rceil - 1 = \lceil \frac{n-2}{2} \rceil$.
Note that when $n$ is odd, $(1010...10\LLL 1)$ is a fixed labeling, when $n$ is even, $(1010..1\LLL 0)$ is a fixed labeling.
\end{corollary}

\begin{proof}
Follows from the proof of Lemma \ref{lem:Cn}.
\end{proof}

\begin{corollary}
$\Cls{\CC_n} = n+1$.
\end{corollary}
\begin{proof}
From Lemma \ref{lem:Cn} we have
$ \Cls{ \CC_n } = \Cls{ \AA_{n-1} } + \Cls{ \AA_{n-1}^{(n-1)} } $.
Applying the results we have for $\Cls{ \AA_{n-1} }$ and $\Cls{ \AA_{n-1}^{(n-1)} }=\Cls{ \AA_{n-1}^{(1)} }$, we get
\begin{eqnarray*}
\Cls{  \CC_n } & = & \Cls{ \AA_{n-1} } + \Cls{  \AA_{n-1}^{(n-1)} } = \\
& = &
\begin{cases}
((k-1)+2) + ((k-1)+1) = 2k+1 = n+1 & \mbox{ if } n=2k \ , \\
(k+1) + (k+1) = 2k+2 = n+1 & \mbox{ if } n=2k+1 \ .
\end{cases}
\end{eqnarray*}
Note that it follows also from Corollary \ref{cor:Cn}(ii).
\end{proof}

\subsection{The diagram $\tilde{\CC}_n$} \label{subsec:tilde-Cn}

The affine Dynkin diagram of of type $\tilde{\CC}_n$ with $n \ge 3$, denoted $\CC_n^{(1)}$ in \cite[Table 6]{OV}, is
\begin{equation*}
\sxymatrix{ \bc{0} \arr & \bc{1} \rline & \cdots & \lline \bc{n-1} & \all \bc{n} }
\end{equation*}
where the end vertices are the longer vertices. The longer vertices do not see the short vertices and
therefore are not affected by them (see formula \eqref{eq:move-nonsimple} in section \ref{sec:rules}). As before, we introduced a short notations: $c_0 \RRR c_1 \cdots c_{n-1} \LLL c_n$ for a labeling $\cc$.

\begin{lemma} \label{lem:tilde-Cn-invariant-ends}
Vertices 0 and $n$ cannot be changed by moves. That is, $\T_0(\cc) = \cc$ and $\T_n(\cc) = \cc$. This implies that two labelings $\cc$ and $\cc'$ can be equivalent only if $c_0 = c'_0$ and $c_n = c'_n$.
\end{lemma}
\begin{proof}
Follows immediately from \eqref{eq:move-nonsimple}.
\end{proof}

\begin{lemma}
We have
\begin{eqnarray} \label{eq:tilde-Cn-L-decom}
L(\tilde{\CC}_n) & = & \{ 0 \RRR \aa \LLL 0 \mid \aa \in L(\AA_{n-1}) \} \sqcup  \{ 0 \RRR \aa \LLL 1 \mid \aa \in L(\AA_{n-1}) \} \sqcup \nonumber \\ & & \sqcup  \{ 1 \RRR \aa \LLL 0 \mid \aa \in L(\AA_{n-1}) \} \sqcup \{ 1 \RRR \aa \LLL 1 \mid \aa \in L(\AA_{n-1}) \} \nonumber
\end{eqnarray}
and since
\begin{eqnarray*}
\Cl(\AA_{n-1}) \cong \Cl\{ 0 \RRR \aa \LLL 0 \mid \aa \in L(\AA_{n-1}) \}, \\
\Cl(\AA_{n-1}^{(n-1)}) \cong \Cl\{ 0 \RRR \aa \LLL 1 \mid \aa \in L(\AA_{n-1}) \}, \\
\Cl(\AA_{n-1}^{(1)}) \cong \Cl\{ 1 \RRR \aa \LLL 0 \mid \aa \in L(\AA_{n-1}) \}, \\
\Cl(\AA_{n-1}^{(1,n-1)}) \cong \Cl\{ 1 \RRR \aa \LLL 1 \mid \aa \in L(\AA_{n-1}) \}, \\
\end{eqnarray*}
we get
\begin{equation} \label{eq:tilde-Cn-Or-decom}
\Cl(\tilde{\CC}_n) \cong \Cl(\AA_{n-1}) \sqcup \Cl(\AA_{n-1}^{(n-1)}) \sqcup \Cl(\AA_{n-1}^{(1)}) \sqcup \Cl(\AA_{n-1}^{(1,n-1)})
\end{equation}
\end{lemma}

\begin{proof}
The 4 bijections in the middle of the lemma are true by construction as one can easily see the labels 1 on the longer vertices $0$ and $n$ are the same as boxed 1's. I.e. they cannot be changed by moves and their neighbors affected by them in the same way. Actually, the fact that the longer vertex with label 1 acts as a boxed 1 is established in the case for $\CC_n$, cf. Section \ref{subsec:Cn}.

Now we define a map
\[
\varphi : L(\tilde{\CC}_n) \longrightarrow L(\AA_{n-1}) \sqcup L(\AA_{n-1}^{(n-1)}) \sqcup L(\AA_{n-1}^{(1)}) \sqcup L(\AA_{n-1}^{(1,n-1)})
\]
as follows:
\begin{itemize}
\item If $\cc = ( 0\RRR\aa\LLL0 )$ then $\varphi(\cc)=\aa \in L(\AA_{n-1})$.
\item If $\cc = ( 0\RRR\aa\LLL1 )$ then $\varphi(\cc)=\aa\ll\boe \in L(\AA_{n-1}^{(n-1)})$.
\item If $\cc = ( 1\RRR\aa\LLL0 )$ then $\varphi(\cc)=\boe\ll\aa \in L(\AA_{n-1}^{(1)})$.
\item If $\cc = ( 1\RRR\aa\LLL1 )$ then $\varphi(\cc)=\boe\ll\aa\ll\boe \in L(\AA_{n-1}^{(1,n-1)})$.
\end{itemize}
This map is clearly a bijection. The quotient map
\[
\varphi_* : \Cl(\tilde{\CC}_n) \longrightarrow \Cl(\AA_{n-1}) \sqcup \Cl(\AA_{n-1}^{(n-1)}) \sqcup \Cl(\AA_{n-1}^{(1)}) \sqcup \Cl(\AA_{n-1}^{(1,n-1)})
\]
defined by $\varphi_*[\cc] = [ \varphi(\cc) ]$ is well-defined: Lemma \ref{lem:tilde-Cn-invariant-ends} ensures us that for each summand in the disjoint union, the passage from labels to classes (i.e. taking the equivalent class) sends $L(-)$ to $\Cl(-)$. The map $\varphi_*$ is a bijection as well (the inverse map is $\varphi_*^{-1}([\cc']) = [\varphi^{-1}(\cc')]$) and the result follows.
\end{proof}

\begin{proposition}
$ \Cls{\tilde{\CC}_n} = 2n+2 = 2 \cdot (n+1) $.

We can take as \minreps the following labelings:
\begin{eqnarray*}
0\RRR\xi_i^{n-1}\LLL0 & \mbox{ for } & i=0,..., \lceil(n-1)/2 \rceil \ , \\
0\RRR\eta_i^{n-2}0\LLL1 & \mbox{ for } & i=0,..., \lceil(n-2)/2 \rceil \ , \\
1\RRR0\xi_i^{n-2}\LLL0 & \mbox{ for } & i=0,..., \lceil(n-2)/2 \rceil \ , \\
1\RRR0\xi_i^{n-3}0\LLL1 & \mbox{ for } & i=0,..., \lceil(n-3)/2 \rceil \ , \\
& \mbox{ and } & 1\RRR1...1\LLL1 \quad (\mbox{the label with all 1's}).
\end{eqnarray*}
\end{proposition}

\begin{proof}
From the above lemma and Proposition \ref{prop:aux-Cn} we have
\begin{eqnarray*}
\Cls{\tilde{\CC}_n} & = & \Cls{\AA_{n-1}} + \Cls{\AA_{n-1}^{(n-1)}} + \Cls{\AA_{n-1}^{(1)}} + \Cls{\AA_{n-1}^{(1,n-1)}} \\
& = & \left( \left\lceil \frac{n-1}{2} \right\rceil + 1 \right) + \left( \left\lceil \frac{n-2}{2} \right\rceil + 1 \right) + \left( \left\lceil \frac{n-2}{2} \right\rceil + 1 \right) + \left( \left\lceil \frac{n-3}{2} \right\rceil + 2 \right)
\end{eqnarray*}
It is easiest to check for odd and even $n$ separately.
\begin{itemize}
\item \textbf{For $n = 2k+1$:} then $n-1 = 2k$, $(n-1)-1 = n-2 = 2k-1$, $(n-1)-2 = n-3 = 2k-2 = 2(k-1)$ and
\begin{eqnarray*}
\Cls{\tilde{\CC}_n} & = & (\frac{2k}{2}+1) + (\frac{2k}{2}+1) + (\frac{2k}{2}+1) + (\frac{2k-2}{2}+2) \\ & = & 3k+3 + k-1 + 2 = 4k + 4 = 2(2k+1)+2 = 2n+2 = 2(n+1)
\end{eqnarray*}
\item \textbf{For $n = 2k$:} then $n-1 = 2k-1$, $(n-1)-1 = n-2 = 2k-2=2(k-1)$, $(n-1)-2 = n-3 = 2k-3 = 2(k-1)-1$ and
\begin{eqnarray*}
\Cls{\tilde{\CC}_n} & = & (\frac{2k}{2}+1) + (\frac{2k-2}{2}+1) + (\frac{2k-2}{2}+1) + (\frac{2k-2}{2}+2) \\ & = & 3k+1 + k-1 + 2 = 4k + 2 = 2(2k)+2 = 2n+2 = 2(n+1)
\end{eqnarray*}
\end{itemize}
So in both cases we got that the number of classes is $2n+2 = 2 \cdot (n+1)$.
\end{proof}

\subsection{Diagrams of type $\DD_n$} \label{subsec:Dn}
\label{sec:Dn}

\expBE

The Dynkin diagram of type $\DD_n$ with $n \ge 4$ is
\begin{equation*}
\sxymatrix{ \bc{1} \rline & \cdots \rline & \bc{n-3} \rline & \bc{n-2} \dline \rline & \bc{n-1} \\
& & & \bcu{n} & }
\end{equation*}
(this Dynkin diagram corresponds to the compact group $\Spin(2n)$ which is the universal cover of the special orthogonal group $\SO(2n)$).
This diagram has a vertex of degree 3, the vertex $n-2$. Throughout Section \ref{sec:Dn} we assume $n \ge 4$.

\begin{notation}
For brevity we introduce the following short notation:
\[
d_1 ... d_{n-2} \frac{d_{n-1}}{d_n} :=  \left( \sxymatrix{ d_1 \rline & \cdots & d_{n-2} \lline \dline \rline & d_{n-1} \\ & & d_n & } \right) \ .
\]
\end{notation}

\begin{remark}
Let us note that other than $\ell_0^{(0)} = \xi_0 \frac{0}{0} = 0...0\frac{0}{0}$ we have the following fixed labelings:
\begin{itemize}
\item We always have the fixed labeling $\ell_2^{(0)} = 0...0\frac{1}{1} = \xi_0^{n-2}\frac{1}{1}$.
\item If $n=2k$ we have another 2 fixed labelings
$$  \ell_1^{(0)} = 101...010\frac{1}{0} = \xi_{k-1}^{n-2}\frac{1}{0} \ , \quad\ell_3^{(0)} = 101...010\frac{0}{1} = \xi_{k-1}^{n-2}\frac{0}{1} \ . $$
\end{itemize}
Each fixed labeling is its own orbit, that is $[\ell_i^{(0)}] = \{ \ell_i^{(0)} \} \in \Or(\DD_n)$.
\end{remark}

\begin{lemma} \label{lem:Dn-basics2}
For $\DD_n$ with $n \ge 4$, if a labeling $\dd$ of $\DD_n$ is not a fixed labeling then $\dd$ can be changed by moves to $\dd'$ such that $d'_{n-1}=d'_n=0$.
\end{lemma}

\begin{proof}

Let $\dd = \aa \frac{\vk}{\lambda}$ with $\aa \in L(\AA_{n-2})$ and $\vk, \lambda \in \{0,1\}$. We show that either $\dd \sim \dd' = \aa'\frac{0}{0}$ or $\dd$ is a fixed labeling.

If $d_{n-2}=1$ the clearly $(...1\frac{\vk}{\lambda}) \sim (...1\frac{0}{0})$ by applying $\T_n$ if $\lambda=1$ and $\T_{n-1}$ if $\vk=1$. Now suppose $d_{n-2} = 0$. If $\vk=\lambda=0$ we have nothing to prove. Assume $\vk=\lambda=1$. Then $\dd = (...0\frac{1}{1})$. If $\aa \ne 0$ then there is at least one component, we push it right to get $(...0\frac{1}{1}) \sim (...10\frac{1}{1})$ and then unsplit by applying $\T_{n-2}$ and get $(...10\frac{1}{1}) \sim (...11\frac{1}{1}) \sim (...11\frac{0}{0})$. Otherwise, if $\aa = 0$ then $\dd = \ell_2^{(0)}$.
Assume $\vk=1$ and $\lambda=0$ (the other case is treated similarly). Then $\dd = (...0\frac{1}{0})$. To change $d_{n-1}=\vk=1$ to 0 we must push it to vertex $n-2$. This cannot be done only if $n=2k$ is even and $\aa = \xi_{k-1}^{n-2} = (10...10)$. Then $\dd = (10...10\frac{1}{0}) = \ell_1^{(0)}$.

The lemma is proved.
\end{proof}


\begin{proposition} \label{lem:orbits-Dn} \label{prop:orbits-Dn}
For $\DD_n$ with $n \ge 4$ we construct a map
\[
\varphi : \begin{array}{rcl} L(\AA_{n-2}) & \longrightarrow & L(\DD_n) \\
\a & \longmapsto & \a \frac{0}{0}
\end{array}
\]
Then the induces map $\varphi_* : \Or(\AA_{n-2}) \to \Or(\DD_n)$ defined by $\varphi_*([\a]) = [ \varphi(\a)]$ is well-defined and injective. Moreover,
\begin{enumerate}
\item[(i)] In the case $n=2k$, when there are 3 fixed labelings, we have
\[ \im \varphi_* =  \Or(\DD_n) \smallsetminus \{
[\ell_2^{(0)}] , \ [\ell_1^{(0)}], \ [\ell_3^{(0)}] \} = \Omega_e \ .
\]
\item[(ii)] In the case $n=2k+1$, we have
\[ \im \varphi_* = \Or( \DD_n ) \smallsetminus \{ [ \ell_2^{(0)} ]\} = \Omega_o \ . \]
\end{enumerate}

%
\end{proposition}

\begin{proof}
The map $\varphi_*$ is well-defined. For $i=1,...,n-3$ we have $\T_i^D = \T_i^A$. For $i=n-2$ we have for $\Img \varphi$ that $\T_{n-2}^D = \T_{n-2}^A$ since $d_n = d_{n+1} = 0$. Clearly, if $\aa \Wsim{A} \aa'$ then $\varphi(\aa) \Wsim{D|_A} \varphi(\aa')$ where $\Wsim{D|_A}$ mean we consider only the move $\T_i^D = \T_i^A , \ i=1,...,n-2$, and this implies $\varphi(\aa) \Wsim{D} \varphi(\aa')$.



The map $\varphi_*$ is also injective. Indeed, suppose $\varphi(\aa) \sim \varphi(\aa')$, then $\aa\frac{0}{0} \sim \aa'\frac{0}{0}$, then $\aa'$ and $\aa$ have the same number of components, which by Theorem \ref{th:basic-An} implies $\aa \sim \aa'$. We showed that if $\varphi_*([\aa]) = \varphi_*([\aa'])$ then $[\aa] = [\aa']$.

Now we prove the assertion about the images. Clearly then $\ell_i^{(0)} \notin \im \varphi$ for $i=1,2,3$ as $d_{n-1}=1$ or $d_n=1$. Since these are fixed labelings each of their orbits contains only one element and is not equivalent to any other labeling. In particular, it is not equivalent to any labeling in the image of $\varphi$. Thus, $[\ell_i^{(0)}] \notin \im \varphi_*$. By Lemma \ref{lem:Dn-basics2}, if $\dd$ is not a fixed labeling then $\dd \sim \aa\frac{0}{0}$ and thus $[\dd] = \varphi_*([\aa])$.

The proposition is proved.
\end{proof}

\begin{corollary} \label{cor:Dn-reps}
\Zerois just the labeling $0 = \ell_0^{(0)} = \xi_0^{n-2}\frac{0}{0}$.
\Reps are:
\begin{itemize}
\item For $n=2k+1$  we can take the following representatives
\[
\xi_0^{n-2} \frac{0}{0} = 0...0\frac{0}{0} \ , \quad \xi_1^{n-2} \frac{0}{0} = 10...0\frac{0}{0} \
, \quad ... , \quad \xi_{k-1}^{n-2}  \frac{0}{0} = 10...100\frac{0}{0} \ ,  \quad \xi_{k}^{n-2}
\frac{0}{0} = 101..01\frac{0}{0}
\]
and the fixed labeling $\ell_2^{(0)} = \xi_0^{n-2} \frac{1}{1} = 0...0\frac{1}{1}$.
\item For $n=2k$ we can take the following $k$ representatives
\[
\xi_0^{n-2} \frac{0}{0} = 0...0\frac{0}{0} \ , \quad \xi_1^{n-2} \frac{0}{0} = 10..0\frac{0}{0} \ ,
\quad ... , \quad \xi_{k-1}^{n-2} \frac{0}{0} = 10..10\frac{0}{0} \ ,
\]
and the 3 fixed labelings (each constitutes an orbit):
\[
\quad \ell_1^{(0)} = \xi_{k-1}^{n-2} \frac{1}{0} = 10..10\frac{1}{0} \ , \quad \ell_3^{(0)} = \xi_{k-1}^{n-2} \frac{0}{1} =
10..10\frac{0}{1} \ , \quad \ell_2^{(0)} = \xi_0^{n-2} \frac{1}{1} = 0..0\frac{1}{1} \ .
\]
\end{itemize}
In all cases $\xi_i \in L(\AA_{n-2})$ (i.e. it includes $n-2$ vertices) for all $0 \le i \le k-1$.
\end{corollary}
\begin{proof}
Follows from Proposition \ref{prop:orbits-Dn}.
\end{proof}

\begin{example}
For $\DD_5$ we have 4 representatives of orbits
\[ 000\frac{0}{0} \ , \quad 100\frac{0}{0} \ , \quad 101\frac{0}{0} , \quad 000\frac{1}{1} \ . \]
For $\DD_6$ we have 6 representatives of orbits
\[
0000\frac{0}{0} \ , \quad 1000\frac{0}{0} \ , \quad 1010\frac{0}{0} \ , \quad 1010\frac{1}{0} \ ,
\quad 1010\frac{0}{1} \ , \quad 0000\frac{1}{1} \ .
\]
\end{example}

\begin{corollary}
For $\DD_n$:
\begin{equation}
\Orbs{\DD_n} = \begin{cases} k+3 & \mbox{ if } n = 2k \ , \\ k + 2 & \mbox{ if } n =
2k+1 \ .
\end{cases}
\end{equation}
\end{corollary}

\begin{proof}
From Proposition \ref{prop:orbits-Dn} we have
\begin{equation*}
\# H^1(\R,\DD_n)=\Orbs{\DD_n} = \begin{cases} \Orbs{\AA_{n-2}} + 3 & \mbox{ if } \ n = 2k \ , \\
\Orbs{\AA_{n-2}} + 1 & \mbox{ if } \ n = 2k+1 \ . \end{cases}
\end{equation*}
From Section \ref{sec:An} we have
\[ \Orbs{\AA_{2k-2}} = \Orbs{\AA_{2(k-1)}} = k-1 + 1 = k \]
and
\[ \Orbs{\AA_{2k+1-2}} = \Orbs{\AA_{2k-1}} = \Orbs{\AA_{2(k-1)+1}} = k-1 + 1 + 1 = k + 1 \ , \]
and then from the formula above
\begin{equation*}
\Orbs{\DD_n} = \begin{cases}  \Orbs{\AA_{n-2}} + 3 = k+3 & \mbox{ if } n = 2k \ , \\ \Orbs{\AA_{n-2}} + 1 = k + 1 + 1 = k+2 & \mbox{ if } n =
2k+1 \ .
\end{cases}
\end{equation*}
\end{proof}

\subsection{Diagrams of type $\tilde{\DD}_n$} \label{subsec:tilde-Dn}

The affine Dynkin diagram of type $\tilde{\DD}_n$ with $n+1$ vertices labeled $\{0,1,...,n\}$ with $n \ge 5$ is
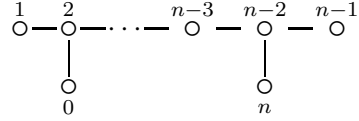
\begin{equation*}
\sxymatrix{ \bc{1} \rline & \bc{2} \dline \rline & \cdots \rline & \bc{n-3} \rline & \bc{n-2} \dline \rline & \bc{n-1} \\
& \bcu{0} & & & \bcu{n} & }
\end{equation*}
It is denoted $\DD_n^{(1)}$ in \cite[Table 6]{OV},
This diagram has two vertices of degree 3, the vertex $2$ and the vertex $n-2$.

We introduce the following notation
\[ \dd = \left( \frac{d_1}{d_0}d_2 ... d_{n-2} \frac{d_{n-1}}{d_n} \right) \]
for a labeling.

\begin{remark} \label{lem:tilde-Dn-fixed-points1}\label{rem:tilde-Dn-fixed-points1}
The labelings
\[
\ell_0 = \left(\frac{0}{0}0...0\frac{0}{0} \right), \ \ell_l = \left(\frac{1}{1}0...0\frac{0}{0} \right), \
\ell_r = \left(\frac{0}{0}0...0\frac{1}{1} \right), \ \ell_c = \left(\frac{1}{1}0...0\frac{1}{1} \right)
\]
are fixed labelings.
\end{remark}

\begin{remark} \label{lem:tilde-Dn-fixed-points2}\label{rem:tilde-Dn-fixed-points2}
If $n=2k$ is even then the labelings
\[
\ell_1 = \left(\frac{1}{0}0\xi_{k-2}^{n-5}0\frac{1}{0} \right), \ \ell_2 = \left(\frac{0}{1}0\xi_{k-2}^{n-5}0\frac{1}{0} \right), \
\ell_3 = \left(\frac{1}{0}0\xi_{k-2}^{n-5}0\frac{0}{1} \right), \ \ell_4 = \left(\frac{0}{1}0\xi_{k-2}^{n-5}0\frac{0}{1} \right)
\]
are fixed labelings. Note that each of these fixed labelings has $(k-2)+2=k=n/2$ components.
\end{remark}

\begin{remark}
Note the symmetry of the fixed labelings. They are due to the reflection symmetries of $\tilde{\DD}_n$.
\end{remark}

\begin{proposition}
If $n=2k$ then
\[ \Cls{\tilde{\DD}_n} = k + 7 = \frac{n}{2} + 7 \ . \]
\end{proposition}
\begin{proof}
We consider the map $\varphi : L(\AA_{n-3}) \to L(\tilde{\DD}_n)$ defined by
\[  (d_2,d_3,...,d_{n-2}) \longmapsto \frac{0}{0}d_2d_3...d_{n-2}\frac{0}{0} \]
and its quotient map $\varphi_* : \Cl(\AA_{n-3}) \to \Cl(\tilde{\DD}_n)$ defined by $\varphi_*([\aa]) = [ \varphi(\aa) ]$.

Then $\varphi_*$ is well-defined, if $\aa' \sim \aa$ then clearly $\varphi(\aa') \sim \varphi(\aa)$.
All the fixed labelings but $\ell_0$ are not in the image of $\varphi$ and thus not in the image of $\varphi_*$ as each forms its own equivalence class.

The map $\varphi_*$ is injective. Indeed, suppose $\varphi(\aa) \sim \varphi(\aa')$, then $\frac{0}{0}\aa\frac{0}{0} \sim \frac{0}{0}\aa'\frac{0}{0}$, then $\aa'$ and $\aa$ have the same number of components, which by Theorem \ref{th:basic-An} implies $\aa \sim \aa'$. We showed that if $\varphi_*([\aa]) = \varphi_*([\aa'])$ then $[\aa] = [\aa']$.

Now we prove that
\[ \Img \varphi_* = \Cl(\tilde{\DD}_n) \smallsetminus \{ [\ell_l], [\ell_r], [\ell_c], [\ell_1], [\ell_2], [\ell_3], [\ell_4] \} \ . \]
The proof is similar to the case of $\DD_n$. We show that if $\dd$ is not a fixed labeling then we can change it by moves to a labeling $\dd'$ with $d'_0=d'_1=d'_{n-1}=d'_n=0$ and thus $\dd' \in \Img \varphi$ so $[\dd] = [\dd'] \in \Img \varphi_*$. Because of the symmetries of the left end and right end and the symmetry of the fixed labelings it is enough to consider only one end. Without loss of generality we shall treat the right end.
\begin{itemize}
\item Assume $d_{n-1}=d_n=1$ in $\dd \in L(\tilde{\DD}_n)$. If there is at least 1 component on one of the vertices $2,...,n-2$ then we can push it to the vertex $n-3$, do unsplitting and then swallow the 1's of the leafs (vertices $n-1$ and $n$). Graphically what we do is
    \[ ...10\frac{1}{1} \overset{\T_{n-2}}{\longmapsto} ...11\frac{1}{1} \overset{\T_{n-1}}{\longmapsto} 11\frac{0}{1} \overset{\T_{n}}{\longmapsto} 11\frac{0}{0} \]
    If $d_2 = ... = d_{n-2} = 0$ then we can't do that but then $\dd = \ell_r$ or $\dd = \ell_c$, which we assumed is not the case, or $d_0=0, \ d_1 = 1$ and then we push the 1 to the right end of the neck and do unsplitting and swallowing.
\item Assume $d_{n-1}=1$ and $d_n=0$ in $\dd \in L(\tilde{\DD}_n)$. The only way we can't swallow or move the 1 in $d_{n-1}$ is to have $d_{n-2} = d_{n-4} = ... = d_2 = 0$ and $d_{n-3} = ... = d_1 = 1$ and $d_0 = 0$ (by symmetry another option is $d_0 = 1$ and $d_1 = 0$) but these are the fixed labelings $\ell_1$ and $\ell_2$ which we assumed is not the case.
\item The case $d_{n-1}=0$ and $d_n=1$ in $\dd \in L(\tilde{\DD}_n)$ is symmetric to the previous case.
\end{itemize}
Thus we showed that every labeling which is not a fixed labeling can be changed by moves to have 0 in the leafs, and thus it is in the image of $\varphi_*$.

We conclude that $\varphi_*$ is an isomorphism onto its image and thus
\[ \Cl(\tilde{\DD}_n) \cong \Cl(\AA_{n-3}) \sqcup \{ [\ell_l], [\ell_r], [\ell_c], [\ell_1], [\ell_2], [\ell_3], [\ell_4] \} \ . \]

From section \ref{sec:An} and Eq. \eqref{eq:An-num-of-classes} we have that
\[
\Cls{\AA_{n-3}} = \left\lceil \frac{n-3}{2} \right\rceil + 1 = \left\lceil \frac{2k-2-1}{2} \right\rceil + 1 = k-1 + 1 = k
\]
and thus
\[ \Cls{\tilde{\DD}_n} = \Cls{\AA_{n-3}} + 7 = k + 7 \ . \]
\end{proof}

For the case $n=2k+1$ we need the following lemma:

\begin{lemma} \label{lem:tilde-Dn-oddcase}
Let $n=2k+1$ and consider $\tilde{\DD}_n$. Then a non-fixed labeling with $k$ components must have a 1 in at least one of the vertices 0,1,$n-1$ or $n$ (``the leafs''). Moreover, all labelings with $k$ components which are not fixed labelings and have $k-1$ components on the vertices $2,3,..,n-2$ (``the neck'') are equivalent. We shall denote this equivalence class by $\kappa$.
\end{lemma}

\begin{proof}[Idea of the proof]
We divide the diagram to neck and leafs. The idea is that the neck is maximally packed with $k-1$ components and we can't do much with it other then shift it as a block leftward or rightward by one vertex, or expand 1 component in it or shrink it back. The possibilities for the component in the leafs are rather limited, also due to symmetries: the symmetry group of the diagram is $S \cong \mathbb{Z}/2\mathbb{Z} \times \mathbb{Z}/2\mathbb{Z} \times \mathbb{Z}/2\mathbb{Z}$ of left leafs reflection, right leafs reflection and left-right reflection, which imply that many different cases are the same up to $s \in S$. So essentially there aren't many cases and the limited number of really different cases can be directly checked to see that they agree with the assertion of the lemma.

In the formal proof we use induction. We first check the simplest case $n=5$, $k=2$ and a direct calculation shows that the lemma holds for this case. Then we show in the induction step that the assertion for $n'=n+2$, $k'=k+1$ follows from the assertion for $n$,$k$ by adding (to $n,k$) or removing (from $n',k'$) two adjunct neck vertices with 1 component. We show that equivalence in the $n$,$k$ case implies equivalence in the $n',k'$ case.
\end{proof}

The following is the detailed proof.

\begin{proof}
The first assertion is clear, since in a neck of length $n-3 = 2k+1-3=2k-2=2(k-1)$ we can put only $k-1$ components, and in that case we can put them such that the leftmost or rightmost vertex is 0 and the remaining component is then put in one of the adjunct leafs.

\medskip

The second assertion is proved by induction.

\textbf{The base case} is $n=5$, $k=2$. Then the diagram is
\begin{equation}
\sxymatrix{ \bc{1} \rline & \bc{2} \dline \rline & \bc{3} \dline \rline & \bc{4} \\
                          & \bcu{0}              & \bcu{5}              &             }
\end{equation}
We see that in vertice 2 and 3 (let's call them ``the neck'') we can put only 1 component ($k-1=1$). So we must put another component in the vertices 0,1,$n-1$ or $n$ (let's call them ``the leafs''). Due to symmetries we may assume without loss of generality that we start from
\[ \frac{0}{0}10\frac{1}{0} \]
and by direct checking one can see that all non-fixed labelings with $k=2$ component are equivalent to it and must have it least 1 component in the leafs. We list here some of the possibilities, all lie in the equivalence class $\kappa$:
\[ \frac{1}{0}00\frac{1}{0} , \quad \frac{1}{0}01\frac{0}{0}, \quad \frac{0}{0}10\frac{0}{1}, \quad \frac{1}{1}10\frac{1}{0} \ ... \]

\textbf{The induction step:} assume we have a diagram $\tilde{\DD}_{n}$ with $n'=2k'+1$ vertices, $n \ge 7$ and $k \ge 3$ such that $n'=n+2$ and $k'=k+1$. Then the neck is of length $n'-3=n-1$. Assume the neck has $k'-1=k$ components and there is 1 component in the leafs. Without loss of generality assume the leftmost neck vertex $a_2 = 1$. Then the neck must be of the form
\[ \frac{\ast}{\ast}101...10\frac{1}{0} \quad \mbox{ or } \quad \frac{\ast}{\ast}101...10\frac{0}{1} \]
and without loss of generality we may assume it is the left form. Note that in these forms the labels of the left leafs do not matter and we may assume that $a_0=a_1=0$ by swallowing.

We reduce the case of $n'=n+2$ to the case of $n$ by removing two adjunct vertices with labels $0\ll1$ from the middle of the neck and contract the diagram. For example:
\begin{equation}
\aa = \frac{\ast}{\ast}1\underline{01}0\frac{\ast}{\ast} \longmapsto  \frac{\ast}{\ast}10\frac{\ast}{\ast} = \bb \ .
\end{equation}
It is clear that if $\bb$ has $k-1$ components in the neck then $\aa$ has $k=k'-1$ components. Moreover, since $\bb$ has at least 1 component in the leafs so does $\aa$.

We want to show that if $\aa \mapsto \bb$ and $\aa' \mapsto \bb'$ by the removal of two adjunct neck vertices and $\bb \sim \bb'$ then $\aa \sim \aa'$. The idea is to treat the component to be removed as a block with the adjunct component (i.e. the neighbor of the vertices removed which has a label 1). We now give a formal proof and after it an example which illustrates the algorithm given in the proof.

Assume that the adjunct vertices $a_j a_{j+1}$ with $2 < j < n-3$ were removed from $\aa$ and suppose that $\bb'$ is obtained from $\bb$ by a move $\T_i$. It is convenient to keep the original numbering of the vertices, i.e. $0,1,2,...,j-1,j+2,...,n-1,n$. We note that:
\begin{itemize}
\item If $i=j,j+1$ we do nothing in $\bb$.
\item For any other $i \ne j-1,j,j+1,j+2$ it is clear that $\bb' = \T_i(\bb) \Rightarrow \aa' = \T_i(\aa)$.
\end{itemize}
As for the cases $i=j-1,j+2$:
\begin{itemize}
\item If $a_j = 1$ and $\bb' = \T_{j-1}(\bb)$ then $\aa' = \T_{j-1}(\aa)$.
\item If $a_j = 0$ and $\bb' = \T_{j-1}(\bb)$ then $\aa' = \T_{j-1} \T_{j} \T_{j+1}(\aa)$.
\item If $a_{j+1}=1$ and $\bb' = \T_{j+2}(\bb)$ then $\aa' = \T_{j+2}(\aa)$.
\item If $a_{j+1}=0$ and $\bb' = \T_{j+2}(\bb)$ then $\aa' = \T_{j+2} \T_{j+1} \T_j (\aa)$.
\end{itemize}
These four cases can be checked directly and they are actually two cases due to symmetry (so it is enough to check the first two cases, for example).

Let us show a simple example. Consider $\aa = a_{j-1} \underline{ a_j a_{j+1}} a_{j+2}$ and its image $\bb = a_{j-1}.a_{j+2}$ (the underlined vertices were removed and the dot stands for their place). Then
\begin{eqnarray*}
\aa = 1\underline{10}1 \overset{\T_j}{\longmapsto} 1\underline{00}1 \overset{\T_{j+1}}{\longmapsto} 1\underline{01}1 \overset{\T_{j+2}}{\longmapsto} 1\underline{01}0 = \aa' \\
\bb = 1.1 \overset{\T_{j+2}}{\longmapsto} 1.0 = \bb'
\end{eqnarray*}
and we see that $\bb'$ is the image of $\aa'$ under the removal of the underlined vertices. Moreover,
\begin{eqnarray*}
\aa' =  1\underline{01}0 \overset{\T_{j+2}}{\longmapsto} = 1\underline{01}1 = \aa'' \\
\bb' = 1.0 \overset{\T_{j+2}}{\longmapsto} 1.1 = \bb''
\end{eqnarray*}
and we see that $\bb''$ is the image of $\aa''$ under the removal of the underlined vertices.

Thus we showed that if  $\aa \mapsto \bb$ and $\aa' \mapsto \bb'$ by the removal of two adjunct neck vertices and $\bb \sim \bb'$ then $\aa \sim \aa'$. This concludes the reduction algorithm.

Now we deduce the assertion for $n'=n+2$:  since all $\bb \in L(\tilde{\DD}_n)$ with $k$ components and $k-1$ components in the neck are equivalent and belongs to $\kappa_n$ (by the induction hypothesis) so are their preimages $\aa \in L(\tilde{\DD}_{n'})$ which have $k'=k+1$ components with $k'-1=k$ components in the neck and they belong to $\kappa_{n'}$.

This concludes the induction step, and the proof.
\end{proof}

\begin{proposition}
If $n=2k+1$ then
\[ \Cls{\tilde{\DD}_n} = k + 4 = \frac{n-1}{2} + 4 \ . \]
\end{proposition}
\begin{proof}
We consider the map $\varphi : L(\AA_{n-3}) \to L(\tilde{\DD}_n)$ defined by
\[  (d_2,d_3,...,d_{n-2}) \longmapsto \frac{0}{0}d_2d_3...d_{n-2}\frac{0}{0} \]
and its quotient map $\varphi_* : \Cl(\AA_{n-3}) \to \Cl(\tilde{\DD}_n)$ defined by $\varphi_*([\aa]) = [ \varphi(\aa) ]$.

Then $\varphi_*$ is well-defined, if $\aa' \sim \aa$ then clearly $\varphi(\aa') \sim \varphi(\aa)$.
All the fixed labelings but $\ell_0$ are not in the image of $\varphi$ and thus not in the image of $\varphi_*$ as each forms its own equivalence class.

The map $\varphi_*$ is injective. Indeed, suppose $\varphi(\aa) \sim \varphi(\aa')$, then $\frac{0}{0}\aa\frac{0}{0} \sim \frac{0}{0}\aa'\frac{0}{0}$, then $\aa'$ and $\aa$ have the same number of components, which by Theorem \ref{th:basic-An} implies $\aa \sim \aa'$. We showed that if $\varphi_*([\aa]) = \varphi_*([\aa'])$ then $[\aa] = [\aa']$.

Now we prove that
\[ \Img \varphi_* = \Cl(\tilde{\DD}_n) \smallsetminus \{ [\ell_l], [\ell_r], [\ell_c], \kappa \}  \ . \]
The proof is similar to the case of $\DD_n$. We show that if $\dd$ is not a fixed labeling or in $\kappa$ then we can change it by moves to a labeling $\dd'$ with $d'_0=d'_1=d'_{n-1}=d'_n=0$ and thus $\dd' \in \Img \varphi$ so $[\dd] = [\dd'] \in \Img \varphi_*$. Because of the symmetries of the left end and right end and the symmetry of the fixed labelings it is enough to consider only one end. Without loss of generality we shall treat the right end.
\begin{itemize}
\item Assume $d_{n-1}=d_n=1$ in $\dd \in L(\tilde{\DD}_n)$. If there is at least 1 component on one of the vertices $2,...,n-2$ then we can push it to the vertex $n-3$, do unsplitting and then swallow the 1's of the leafs (vertices $n-1$ and $n$). Graphically what we do is
    \[ ...10\frac{1}{1} \overset{\T_{n-2}}{\longmapsto} ...11\frac{1}{1} \overset{\T_{n-1}}{\longmapsto} 11\frac{0}{1} \overset{\T_{n}}{\longmapsto} 11\frac{0}{0} \]
    If $d_2 = ... = d_{n-2} = 0$ then we can't do that but then $\dd = \ell_r$ or $\dd = \ell_c$, which we assumed is not the case, or $d_0=0, \ d_1 = 1$ and then we push the 1 to the right end of the neck and do unsplitting and swallowing.
\item Assume $d_{n-1}=1$ and $d_n=0$ in $\dd \in L(\tilde{\DD}_n)$. If vertices $2,...,n-2$ has less than $k-1$ components we can push the component of the leaf $d_{n-1}$ to the neck of the vertices $2,...,n-2$. Note that we can't have an alternating labels that form a fixed labeling due to parity considerations (which is why we do not have the fixed labelings $\ell_1,\ell_2,\ell_3,\ell_4$). If there are exactly $k-1$ components in the vertices $2,...,n-3$ ($d_{n-2}=0$) then we are in case of Lemma \ref{lem:tilde-Dn-oddcase} which shows us that $\dd \in \kappa \notin \Img \varphi_*$, which we assumed is not the case. If we have $d_{n-2}=1$ we can simply swallow the 1 of $d_{n-1}$ by applying $\T_{n-1}$.
\item The case $d_{n-1}=0$ and $d_n=1$ in $\dd \in L(\tilde{\DD}_n)$ is symmetric to the previous case.
\end{itemize}
Thus we showed that every labeling which is not a fixed labeling or in $\kappa$ can be changed by moves to have 0 in the leafs, and thus it is in the image of $\varphi_*$.

We conclude that $\varphi_*$ is an isomorphism onto its image and thus
\[ \Cl(\tilde{\DD}_n) \cong \Cl(\AA_{n-3}) \sqcup \{ [\ell_l], [\ell_r], [\ell_c], \kappa \} \ . \]

From section \ref{sec:An} and Eq. \eqref{eq:An-num-of-classes} we have that
\[
\Cls{\AA_{n-3}} = \left\lceil \frac{n-3}{2} \right\rceil + 1 = \left\lceil \frac{2k+1-3}{2} \right\rceil + 1 = k-1 + 1 = k
\]
and thus
\[ \Cls{\tilde{\DD}_n} = \Cls{\AA_{n-3}} + 4 = k + 4 \ . \]
\end{proof}

To conclude:

\begin{corollary}
For $\tilde{\DD}_n$ we have
\[ \Cls{\tilde{\DD}_n} = \begin{cases} k+7 & \mbox{ if  }\ n=2k \ , \\ k+4 & \mbox{ if  }\ n=2k+1 \ . \end{cases} \]
As \minreps we can take $\ell_l, \ell_r, \ell_c$ and $\frac{0}{0} \xi_i^{n-3} \frac{0}{0}$ for $i=0,...,k-1$,
and in addition: $\frac{0}{0}\xi_{k-1}^{n-4}0\frac{1}{0} \in \kappa$ when $n=2k+1$ and $\ell_1, \ell_2, \ell_3, \ell_4$ when $n=2k$.
\end{corollary}

Let us show two examples:

\begin{example}
For $\tilde{\DD}_7$ the diagram (with 8 vertices) is
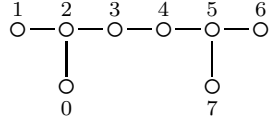
\begin{equation*}
\sxymatrix{ \bc{1} \rline & \bc{2} \dline \rline & \bc{3} \rline &  \bc{4} \rline & \bc{5} \dline \rline & \bc{6} \\
& \bcu{0} & & & \bcu{7} & }
\end{equation*}
and the representatives are:
\[
\frac{0}{0}0000\frac{0}{0} , \quad \ell_l = \frac{1}{1}0000\frac{0}{0}, \quad \ell_r = \frac{0}{0}0000\frac{1}{1}, \quad \ell_c = \frac{1}{1}0000\frac{1}{1},
\]
and
\[
\frac{0}{0}1000\frac{0}{0}, \quad \frac{0}{0}1010\frac{0}{0}, \quad \mbox{ and } \quad \kappa \ni \frac{0}{0}1010\frac{1}{0} .
\]
In total we have 7 classes. Here $n=7$ and $k=3$.
\end{example}

\begin{example}
For $\tilde{\DD}_6$ the diagram (with 7 vertices) is
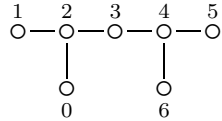
\begin{equation*}
\sxymatrix{ \bc{1} \rline & \bc{2} \dline \rline &  \bc{3} \rline & \bc{4} \dline \rline & \bc{5} \\
& \bcu{0} & & \bcu{6} & }
\end{equation*}
and the representatives are:
\[
\frac{0}{0}000\frac{0}{0} , \quad \ell_l = \frac{1}{1}000\frac{0}{0}, \quad \ell_r = \frac{0}{0}000\frac{1}{1}, \quad \ell_c = \frac{1}{1}000\frac{1}{1},
\]
and
\[
\frac{0}{0}100\frac{0}{0}, \quad \frac{0}{0}101\frac{0}{0},
\]
and
\[
\ell_1 = \frac{1}{0}010\frac{1}{0}, \quad \ell_2 = \frac{0}{1}010\frac{1}{0}, \quad \ell_3 = \frac{1}{0}010\frac{0}{1}, \quad \mbox{ and } \quad \ell_4 =  \frac{0}{1}010\frac{0}{1}.
\]
In total we have 10 classes. Here $n=6$ and $k=3$.
\end{example}

\subsection{Diagram with 3 lengths of vertices}
\label{sec:Xn}

Here we consider the diagram denoted by $A_{2\ell}^{(2)}$ in \cite[Table 6]{OV}. We shell denote it by $\XX_{n}$ and recall it has $n+2$ vertices. The diagram is
\begin{equation*}
\sxymatrix{ \bc{0} \arr & \bc{1} \rline & \cdots & \lline \bc{n} \arr & \bc{n+1} }
\end{equation*}
Recall that a (longer) vertex does not see its shorter neighbor(s).

We shall denote a labeling $\bb = (b_0 \RRR b_1 , ... , b_n \RRR b_{n+1})$ or $\bb = (b_0 \RRR b_1  ...  b_n \RRR b_{n+1})$.

\begin{lemma} \label{lem:X-bijection}
For $\XX_{n}$ we have $\Cl(\XX_{n}) \cong \Cl(\BB_{n+1}) \sqcup \Cl(\BB_{n+1}^{(1)})$ where
\[
B_{n+1}^{(1)} = \sxymatrix{ \boxone \rline & \bc{1} \rline & \cdots & \lline \bc{n} \arr & \bc{n+1} }
\]
\end{lemma}
\begin{proof}
The map $\varphi : L(\XX_{n}) \to L(\BB_{n+1}) \sqcup L(\BB_{n+1}^{(1)})$ defined by
\[
\varphi(b_0 \RRR b_1,...,b_n \RRR b_{n+1}) = \begin{cases}
b_1...b_n \RRR b_{n+1} \in L(\BB_{n+1}) & \mbox{ if } b_0 = 0 , \\
\boe \ll b_1...b_n \RRR b_{n+1} \in L(\BB_{n+1}^{(1)}) & \mbox{ if } b_0 = 1 .
\end{cases}
\]
is clearly a bijection. The quotient map
\[
\varphi_* : \Cl(\XX_{n}) \to \Cl(\BB_{n+1}) \sqcup \Cl(\BB_{n+1}^{(1)}), \qquad \varphi_*([\bb]) = [\varphi(\bb)] \ ,
\]
is well-defined and a bijection as well (cf. subsection \ref{subsec:Cn}).
\end{proof}

We now state a result from \cite{BE-real} about the diagram $\BB_{n+1}^{(1)}$ which is a special cases of $\BB_{n+1}^{(m)}$ with $m=1$.

\begin{proposition} \label{prop:Bn^(1)-reps}
For $\BB_{n+1}^{(1)}$ we can take the following representatives:
\[
\boe \ll \eta_0^n \RRR 0 , \quad \boe \ll \eta_1^n \RRR 0 , \quad ... \quad \boe \ll \eta_k^n \RRR 0 , \quad
\]
where $k = \lceil (n-1)/2 \rceil$ and thus
\begin{equation}
\Cls{\BB_{n+1}^{(1)}} = 1 + \left\lceil \frac{n-1}{2} \right\rceil =
\begin{cases}  1 + k & \mbox{ if } n=2k , \\ 1 + k & \mbox{ if } n=2k+1 \ . \end{cases}
\end{equation}
\end{proposition}

From section \ref{sec:Bn} we have that $\Cls{\BB_{n+1}} = 2 + \left\lceil \frac{(n+1)-1}{2} \right\rceil = 2 + \lceil n/2 \rceil$.

\begin{proposition}
For $\XX_{n}$ we can take as \minreps
\[
0 \RRR \xi_0 \RRR 1 \ , \quad 0 \RRR \xi_0 \RRR 0 , \quad 0 \RRR \xi_1 \RRR 0 \ , \quad ...  ,  \quad 0 \RRR \xi_r \RRR 0
\]
where
\[
r = \left\lceil \frac{(n+1)-1}{2} \right\rceil = \left\lceil \frac{n}{2} \right\rceil  =
\begin{cases} k & \mbox{ if } n=2k \ , \\ k+1 & \mbox{ if } n=2k+1 \ , \end{cases}
\]
and
\[
1 \RRR \eta_0^n \RRR 0 , \quad 1 \RRR \eta_1^n \RRR 0 , \quad ... \quad 1 \RRR \eta_k^n \RRR 0 , \quad
\]
where $k = \lceil (n-1)/2 \rceil$. In total, $\Cls{\XX_{n+2}} = n+3$.
\end{proposition}

\begin{proof}
The assertion about the representatives follows from Lemma \ref{lem:X-bijection}, Corollary \ref{cor:Bn-reps} and Proposition \ref{prop:Bn^(1)-reps}. Then by Lemma \ref{lem:X-bijection} we have
\begin{eqnarray*}
\Cls{\XX_{n}} & = & \Cls{\BB_{n+1}} + \Cls{\BB_{n+1}^{(1)}} = \\
& = &
\begin{cases}
k+2 + (1+k) = 2k+3  = n+3 & \mbox{ if } n=2k \ , \\
k+3 + (1+k) = 2k+4  = n+3 & \mbox{ if } n=2k+1
\end{cases}
= n+3  .
\end{eqnarray*}
\end{proof}

\begin{example}
For $n=1$ the diagram is
\[
\XX_1 = ( \  \sxymatrix{ \bc{0} \arr & \bc{1} \arr & \bc{2} } \ )
\]
and we have 4 classes: the fixed labeling $[0] = \{ (0\RRR0\RRR0) \}$, the fixed labeling $[\ell_1] = \{  (0\RRR0\RRR1) \}$, the class
$ \{ (0\RRR1\RRR0), \ (0\RRR1\RRR1) \}$ and the class \[ \{ (1\RRR0\RRR0), \  (1\RRR1\RRR0), \ (1\RRR1\RRR1), \ (1\RRR0\RRR1) \} \ . \]
We see that indeed $\Cls{\XX_1} = 1 + 3 = 4$.
\end{example}

\subsection{The 4-edged diagram}

Here we consider the diagram denoted by $A_{2}^{(2)}$ in \cite[Table 6]{OV}. The diagram is
\begin{equation*}
\sxymatrix{ \bc{1} \ar@{=>}[r]^{4} & \bc{2} }
\end{equation*}
and by Definition \ref{def:elem-moves-nonsimple} we see it behaves the same as $\BB_2$. Thus the classes are $\{ 0>0 \}$, $\{ 0>1 \}$ and $\{ (1>0 , 1>1) \}$ (here $a_1 > a_2$ denotes $a_1 \overset{4}{\RRR} a_2$).

\subsection{The diagram of longer vertex and opposite leafs}
\label{sec:Yn}

Here we consider the diagram denoted by $A_{2\ell-1}^{(2)}$ in \cite[Table 6]{OV}. We shell denote it by $\YY_{n}$ and recall it has $n+1$ vertices. The diagram is
\begin{equation*}
\sxymatrix{ \bc{n-1} \rline & \bc{n-2} \dline \rline & \cdots & \lline \bc{1} & \all  \bc{0} \\ & \bcu{n} & & & }
\end{equation*}
We shall denote a labeling $\dd = (\frac{d_{n-1}}{d_n}d_{n-2}  ...  d_1 \LLL d_{0})$.

\begin{lemma} \label{lem:Y-bijection}
For $X_{n}$ we have $\Cl(\YY_{n}) \cong \Cl(\DD_{n}) \sqcup \Cl(\DD_{n}^{(1)})$ where
\[
D_{n}^{(1)} = \sxymatrix{ \boxone \rline & \bc{1} \rline & \cdots & \lline \bc{n-2} \dline \rline & \bc{n-1}
                        \\               &               &        &       \bcu{n}                 &        }
\]
\end{lemma}
\begin{proof}
The map $\varphi : L(\YY_{n}) \to L(\DD_{n}) \sqcup L(\DD_{n}^{(1)})$ defined by
\[
\varphi(\frac{d_{n-1}}{d_n}d_{n-2}  ...  d_1 \LLL d_{0}) = \begin{cases}
(\frac{d_{n-1}}{d_n}d_{n-2}  ...  d_1) \in L(\DD_{n}) & \mbox{ if } d_0 = 0 , \\
(\frac{d_{n-1}}{d_n}d_{n-2}  ...  d_1 \ll \boe) \in L(\DD_{n}^{(1)}) & \mbox{ if } d_0 = 1 .
\end{cases}
\]
is clearly a bijection. The quotient map $$\varphi_* : \Cl(\YY_{n}) \to \Cl(\DD_{n}) \sqcup \Cl(\DD_{n}^{(1)})$$ is well-defined and a bijection as well (cf. subsection \ref{subsec:Cn}).
\end{proof}

We now state a result from \cite{BE-real} about the diagram $\DD_{n}^{(1)}$ which is a special cases of $\DD_{n}^{(m)}$ with $m=1$.

\begin{proposition} \label{prop:Dn^(1)-reps}
For \minreps of $\DD_{n}^{(1)}$ we can take:

If $n=2k$ is even
\[
\boe \ll 0\xi_0^{n-3}\frac{0}{0} , \quad \boe \ll 0 \xi_1^{n-3}\frac{0}{0} , \quad ... \quad \boe \ll 0 \xi_{k-1}^{n-3}\frac{0}{0} ,
\]
If $n=2k+1$ is odd
\[
\boe \ll 0\xi_0^{n-3}\frac{0}{0} , \quad \boe \ll 0 \xi_1^{n-3}\frac{0}{0} , \quad ... \quad \boe \ll 0 \xi_{k-1}^{n-3}\frac{0}{0} , \quad
\boe \ll 0 \xi_{k-1}^{n-3}\frac{1}{0} , \quad \boe \ll 0 \xi_{k-1}^{n-3}\frac{0}{1}
\]

Thus, the number of classes is
\begin{equation}
\Cls{\DD_{n}^{(1)}} = \begin{cases}  k & \mbox{ if } \ n=2k , \\ k+2 & \mbox{ if } \ n=2k+1 \ . \end{cases}
\end{equation}
\end{proposition}

From section \ref{sec:Dn} we have that $\Cls{\DD_{n}} = \begin{cases}  k+3 & \mbox{ if } \ n=2k , \\ k+2 & \mbox{ if } \ n=2k+1 \ . \end{cases}$.

\begin{proposition}
For $\YY_{n}$ we can take as \minreps the inverse images under $\varphi$ of the labeling of Proposition \ref{prop:Dn^(1)-reps} and Corollary \ref{cor:Dn-reps}.
In total, $\Cls{\YY_n} = n+3$.
\end{proposition}

\begin{proof}
The assertion about the representatives follows from Lemma \ref{lem:Y-bijection}, Corollary \ref{cor:Dn-reps} and Proposition \ref{prop:Dn^(1)-reps}. Then by Lemma \ref{lem:Y-bijection} we have
\begin{eqnarray*}
\Cls{\YY_{n}} & = & \Cls{\DD_{n}} + \Cls{\DD_{n}^{(1)}} = \\
& = &
\begin{cases}
(k+3) + k = 2k+3  = n+3 & \mbox{ if } n=2k \ , \\
(k+2) + (k+2) = 2k+4  = n+3 & \mbox{ if } n=2k+1
\end{cases}
= n+3  .
\end{eqnarray*}
\end{proof}

\subsection{The double-short diagram}
\label{sec:Zn}

Here we consider the diagram denoted by $D_{\ell+1}^{(2)}$ in \cite[Table 6]{OV}. We shell denote it by $\ZZ_{n}$ and recall it has $n+1$ vertices ($n \ge 2$). The diagram is
\begin{equation*}
\sxymatrix{ \bc{0} & \all \bc{1} \rline & \cdots & \lline \bc{n-1} \arr & \bc{n}  }
\end{equation*}
We shall denote a labeling by $\bb = (b_0\LLL b_1 ... b_{n-1} \RRR b_n)$. Recall that the longer vertices do not see the shorter vertices $0$ and $n$.

\begin{remark}
The following labelings are fixed labelings:
\[
\ell_0 = (0\LLL0...0\RRR0) \ , \quad \ell_l = (1\LLL0...0\RRR0) \ , \quad \ell_r = (0\LLL0...0\RRR1) \ , \quad \ell_c = (1\LLL0...0\RRR1) \ .
\]
\end{remark}

\begin{lemma} \label{lem:Z-bijection}
Consider the map $\varphi : L(\AA_{n-1}) \to L(\ZZ_n)$ defined by
\[ L(\AA_{n-1}) \ni \aa \longmapsto ( 0\LLL\aa\RRR 0 ) \ . \]

Then the quotient map
$\varphi_* : \Cl(\AA_{n-1}) \to \Cl(\ZZ_n)$ is well-defined, injective, and its image is
\[ \Img \varphi_* = \Cl(\ZZ_n) \smallsetminus \{ [\ell_l], [\ell_r], [\ell_c] \} \ . \]
\end{lemma}

\begin{proof}
Clearly $\varphi_*$ is well-defined, if $\aa \sim \aa'$ then clearly $(0\LLL\aa\RRR0) \sim (0\LLL\aa'\RRR0)$. It is also clearly injective: assume $\varphi_*([\aa]) = \varphi_*([\aa'])$. Then $(0\LLL\aa\RRR0) \sim (0\LLL\aa'\RRR0)$. This implies that $\aa$ and $\aa'$ have the same number of components and hence by Theorem \ref{th:basic-An} $\aa \sim \aa'$.

The fixed labelings $\ell_l$, $\ell_r$ and $\ell_c$ are not in the image of $\varphi$ and since each of them forms its own class, the classes $[\ell_l]$, $[\ell_r]$ and $[\ell_c]$ are not in $\Img \varphi_*$. It remain to show that any other labeling is equivalent to a labeling that is in the image of $\varphi$, this will show that all the other classes are in $\Img \varphi_*$. The assertion then follows from the Swallowing Lemma (Lemma \ref{lem:Bn-shortroot-shrink}), applied to the two shorter vertices, since if $\bb \ne \ell_l , \ell_r, \ell_c$ then it must have at least 1 component in the vertices $1$ to $n-1$.
\end{proof}

\begin{proposition}
For $\ZZ_n$ we can take as \minreps the labelings
\[ \ell_l, \ \ell_r , \ \ell_c , \ \ell_0 = (0\LLL \xi_0^{n-1} \RRR 0) \ , \quad  (0\LLL \xi_1^{n-1} \RRR 0) \ , \quad ...
\quad  (0\LLL \xi_k^{n-1} \RRR 0) \]
where $k = \lceil (n-1)/2 \rceil$. The number of classes is
\[
\Cls{\ZZ_n} = 3 + (1 + \left\lceil \frac{n-1}{2} \right\rceil) = k + 4 \ .
\]
\end{proposition}

\begin{proof}
Follows immediately from Lemma \ref{lem:Z-bijection}.
\end{proof}

\bigskip

%
%
%
%
%
%

\subsection{The diagrams of type $\EE_6$, $\EE_7$ and $\EE_8$}

The results for the Dynkin diagrams of types $\EE_6$, $\EE_7$ and $\EE_8$ are well known. See e.g. \cite{BE-real}.

\subsection{The diagram of type $\tilde{\EE}_6$}
\label{sec:tilde-E6}

The affine Dynkin diagram  of $\tilde{\EE}_6$, denoted $\EE_6^{(1)}$ in \cite[Table 6]{OV}, is
\begin{equation*}
\sxymatrix{ \bc{1} \rline & \bc{2} \rline & \bc{3} \rline \dline & \bc{4} \rline & \bc{5}  \\
& & \bcu{6} \dline & & \\ & & \bcu{0}  & &}
\end{equation*}

\begin{proposition} \label{prop:tilde-E6}
The diagram $\tilde{\EE}_6$ has 4 classes. The classes are:
\begin{enumerate}
\item[1.] {\emm The class of zero} consisting of 0, which is a fixed labeling.
\item[2.] The class of the fixed labeling
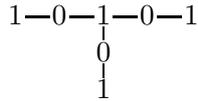
\[ \sxymatrix{ 1 \rline & 0 \rline & 1 \rline \dline & 0 \rline & 1  \\
& & 0 \dline & & \\ & & 1  & &} \]
\item[3.] The class consisting of all labelings with $1$ or $3$  components. 
\item[4.] The class consisting of all labelings with $2$ components.
\end{enumerate}
\end{proposition}

\begin{remark}
The moves in $L(\tilde{\EE}_n)$ for  $n=6,7,8$ preserve the parity of the number of components. See Lemma \ref{lem:vertex3}.
\end{remark}

%
%
%

\subsection{Diagram of type $\tilde{\EE}_7$}
\label{sec:tilde-E7}

The affine Dynkin diagram of type $\tilde{\EE}_7$, denoted $\EE_7^{(1)}$ in \cite[Table 6]{OV},  is
\begin{equation*}
\sxymatrix{  \bc{1} \rline &  \bc{2} \rline &  \bc{3} \rline &  \bc{4}
\rline \dline &  \bc{5} \rline &  \bc{6} \rline & \bc{0} \\
& & & \bcu{7} & & & }
\end{equation*}

\begin{proposition} \label{prop:tilde-E7}
The diagram $\tilde{\EE}_7$ has 6 classes. The classes are:
\begin{enumerate}
\item[1.] {\emm The class of zero} consisting of the fixed labeling zero $\ell_0$.
\item[2.] The fixed labeling $\ell_l$
\[
\sxymatrix{  1 \rline & 0 \rline &  1 \rline & 0
\rline \dline &  0 \rline &  0 \rline  & 0  \\
& & & 1 & & & }
\]
\item[3.] The fixed labeling $\ell_r$
\[
\sxymatrix{  0 \rline & 0 \rline &  0 \rline & 0
\rline \dline &  1 \rline &  0 \rline  & 1  \\
& & & 1 & & & }
\]
\item[4.] The fixed labeling $\ell_c$
\[
\sxymatrix{  1 \rline & 0 \rline &  1 \rline & 0
\rline \dline &  1 \rline &  0 \rline  & 1  \\
& & & 0 & & & }
\]
\item[5.] All the labelings with odd number of components, excluding $\ell_l$ and $\ell_r$.
\item[6.] All the labelings with even number of components, excluding $\ell_c$.
\end{enumerate}
\end{proposition}

%
%
%
%

\subsection{Diagram of type $\tilde{\EE}_8$}
\label{sec:tilde-E8}

The affine Dynkin diagram of type $\tilde{\EE}_8$, denoted $\EE_8^{(1)}$ in \cite[Table 6]{OV}, is
\begin{equation*}
\sxymatrix{ \bc{0} \rline & \bc{1} \rline & \bc{2} \rline & \bc{3} \rline & \bc{4} \rline &
\bc{5} \rline \dline & \bc{6} \rline & \bc{7}  \\ & & & & & \bcu{8} & & }
\end{equation*}

\begin{proposition} \label{prop:tilde-E8}
The diagram $\tilde{\EE}_8$ has 4 classes. The classes are:
\begin{enumerate}
\item[1.] {\emm The class of zero} which contains only 0.
\item[2.] The fixed labeling $\ell_l$
\[
\sxymatrix{ 1 \rline & 0 \rline & 1 \rline & 0 \rline & 1 \rline &
0 \rline \dline & 0 \rline & 0  \\ & & & & & 1 & & }
\]
\item[3.] The class with odd number of components represented by
\[ \sxymatrix{ 1 \rline & 0 \rline & 0 \rline & 0 \rline & 0 \rline & 0 \dline \rline & 0 \rline & 0 \\ & & & & & 0 & & } \ . \]
\item[4.] The class with even number of components excluding $\ell_l$, represented by
\[ \sxymatrix{ 1 \rline & 0 \rline & 1 \rline & 0 \rline & 0 \rline & 0 \dline \rline & 0 \rline & 0 \\ & & & & & 0 & & } \ . \]
\end{enumerate}
\end{proposition}

%
%

\subsection{Diagrams of type $\tilde{\FF}_4$}
\label{subsec:tilde-F4}

The Dynkin diagram of type $\tilde{\FF}_4$, denoted $\FF_4^{(1)}$ in \cite[Table 6]{OV}, is
\begin{equation*}
\sxymatrix{ \bc{1} \rline & \bc{2} & \ar@{=>}[l] \bc{3} \rline & \bc{4} \rline & \bc{0} }\,.
\end{equation*}

\begin{proposition} \label{prop:tilde-F4}
The diagram $\tilde{\FF}_4$ has 4 classes. The classes are:
\begin{enumerate}
\item[1.] {\emm The class of zero} which contains only  $0\ll0\LLL0\ll0\ll0 $.
\item[2.] The class
\[ \left\{ \ 1\ll0\LLL0\ll0\ll0 , \quad   1 \ll 1 \LLL 0\ll 0\ll 0,  \quad   0 \ll 1 \LLL 0 \ll 0 \ll 0 \ \right\} \,. \]
\item[3.] The class represented by $0\ll0\LLL1\ll0\ll0$.
\item[4.] The class represented by $0\ll0\LLL1\ll0\ll1$.
\end{enumerate}
\end{proposition}

\subsection{Diagrams of type $\GG_2$ and $\tilde{\GG}_2$}
\label{subsec:G2}

This case is uninteresting because the triple edge behaves like an undirected single edge (see Section \ref{subsec:non-simply-laced}). Thus the cases of $\GG_2$ and $\tilde{\GG}_2$ are reduced to $\AA_2$ and $\AA_3$ respectively.

The Dynkin diagram of $\GG_2$ is
\begin{equation*}
\sxymatrix{ \bc{1} & \ar@3{->}[l] \bc{2} }\, .
\end{equation*}
because 3 is odd and hence the triple acts behaves like an undirected single edge (see Section \ref{sec:rules}).
The description of classes is  similar to the case  $\AA_2$, because 3 is odd and hence the triple acts behaves like an undirected single edge (see Section \ref{sec:rules}).
We have $\Cls{\GG_2}=2$.
The 2 classes are
\[ [\xi_0] = \{\, 0 \ll 0\, \} \quad \mbox{ and } \quad [\xi_1] = \{\, 1 \ll 0\,,\quad 1 \ll 1\, , \quad 0 \ll 1\, \} \, . \]

The Dynkin diagram of $\tilde{\GG}_2$, denoted $\G_2^{(1)}$ in \cite[Table 6]{OV}, is
\begin{equation*}
\sxymatrix{ \bc{1} & \ar@3{->}[l] \bc{2} \rline & \bc{0} }\, .
\end{equation*}
The description of classes is  similar to the case  $\AA_3$, because 3 is odd and hence the triple acts behaves like an undirected single edge (see Section \ref{sec:rules}).
We have $\Cls{\tilde{\GG}_2}=3$ and the \minreps are $\xi_0$, $\xi_1$ and $\xi_2$. Note that $\xi_0$ and $\xi_2$ are fixed labelings.

\subsection{Exceptional Diagram of 3 short and 2 long vertices}
The diagram denoted in \cite[Table 6]{OV} by $\EE_6^{(2)}$ is
\begin{equation*}
\sxymatrix{ \bc{0} \rline & \bc{1} \rline & \bc{2} & \all \bc{3} \rline & \bc{4} }
\end{equation*}
It has 4 classes:
\begin{enumerate}
\item The fixed labeling $000\LLL00$.
\item The fixed labeling $101\LLL00$.
\item The class of labelings with 1 component in vertices 0-2 and 0's in vertices 3-4, with representative $100\LLL00$.
\item The class of labelings with 1 component in vertices 3-4, with representative $000\LLL01$.
\end{enumerate}

\subsection{Exceptional Diagram of extended $\GG_2$}
The diagram denoted in \cite[Table 6]{OV} by $D_4^{(3)}$ is
\begin{equation*}
\sxymatrix{ \bc{0} \rline & \bc{1} & \ar@3{<-}[l] \bc{2} }
\end{equation*}
By definition \ref{def:elem-moves-nonsimple} we see that the triple edge functions as an undirected single edge. Therefore this diagram behaves exactly as $\AA_3$ and has the same solution.

\section{More general graphs}
\label{sec:general-graphs}

Here we discuss results for more general graphs.

\begin{theorem}[The $\EE_6$ Tree Theorem] \label{th:three-arms}
Led $D$ be a tree that contains the Dynkin diagram $\EE_6$ as a subgraph. Namely: $D$ is a connected simply-laced circuitless graph and with $n \ge 4$ vertices such there exist at least one vertex of degree 3, and this vertex has at least two arms of length greater or equal to 2. For example, for
\begin{equation} \label{eq:prop-three-arms}
\sxymatrix{ \cdots \rline & \bc{1} \rline & \bc{2} \rline & \bc{3} \dline \rline & \bc{4} \rline & \bc{5} \rline & \cdots \\
                          &                &               & \bcu{6} \dline       &               &              &         \\
                          &                &               & \vdots              &               &              &
                                                                                                                               }
\end{equation}
the left arm (vertices 1 and 2) and right arm (vertices 4 and 5) satisfy this condition.

Then in addition to the fixed labelings, $D$ has exactly two classes: the class with labelings with an odd number of components (excluding the fixed labelings) and the class with labelings with an even number of components (excluding the fixed labelings).
\end{theorem}

\begin{remark}
Reeder proved this assertion for $\EE_6$ and later generalized it to a class of trees with condition on their matchings (2005). In Reeder's statement he classify the orbits according to the the quadratic $q(u) = \sum_i u_i^2 + \sum_{i \ll j} u_i u_j \bmod 2$ which equals to the parity of the number of components. Weng (2011) generalized to any tree containing $\EE_6$ but didn't publish his proof.\cite{Weng} We present here a different independent proof, which assumes nothing about a quadratic and requires no knowledge of graph theory. Our proof gives a constructive algorithm to change by a sequence of moves every non-fixed labeling into a labeling with 1 or 2 components, depending on its initial parity.
\end{remark}


Throughout the proofs we shall denote the central branching vertex coming from $\EE_6$ and its label by $a_3$. Before we start, we need the following three lemmas:

\begin{lemma}[Central Vertex Lemma] \label{lem:central-vertex}
If the labeling $\aa$ is not-fixed it is equivalent to a labeling with $a'_3=0$ and also to a labeling with $a''_3=1$. In other words, we can change by a sequence of moves the label of $a_3$ to be 0 or 1, as necessary.
\end{lemma}

\begin{proof}[Proof of the Central Vertex Lemma]
We show that if the label of $a_3$ cannot be changed then the labeling $\aa$ is fixed.

Assume $a_3=0$. If it cannot be changed to 1, then it has an even number of neighboring 1's. Moreover, each of these 1's cannot be pushed backwards (i.e. farther then $a_3$), otherwise we would push 1 component back and have an odd number of neighboring 1's to $a_3$. So these arms are unmoveable. Now we consider the arms with 0's neighboring $a_3$. If they are empty (i.e. all the labels in them are 0) then $\aa$ is a fixed labeling. If not, consider any nonempty arm with 0 neighboring $a_3$. If we can push a component from this arm to the neighbor of $a_3$ then $a_3$ will then have odd number of 1's neighboring, so by unsplitting $\T_3$ we can set $a_3=1$. Thus, in any arm the closest component to $a_3$ cannot be pushed toward $a_3$. As shown in Lemma \ref{lem:clear-path}, this implies that the arm is unmoveable. So we have shown that if $a_3=0$ and it cannot be changed to 1 then all the arms are unmoveable and thus the labeling is fixed.

Assume $a_3=1$. If $a_3$ has an odd number of neighboring 1's we apply $\T_3$ and do splitting, so $a_3=0$ (this argument holds even if $a_3$ has only one neighboring 1, then $\T_3$ is simply shrinking and then $a'_3=0$). If $a_3$ has an even number $m \ge 2 > 0$ of neighboring 1's we choose any neighbor, say $a_i$, and swallow it: we first shrink its component from the farther end (we can do this since the labeling is not-fixed), until we have
\[ (a_3=1) \ll (a_i=1) \ll \equiv 0 \cdots \]
(the $\equiv 0$ means all the neighbors of $a_i$ other than $a_3$ are 0) and in the last step swallow it, by applying $\T_i$. For example:
\[
(a_3=1) \two{1_i}{1}  1  1  0 1 \longmapsto (a_3=1) \two{1_i}{0}  0  0  0 1
\overset{\T_i}{\longmapsto} (a_3=1) \two{0_i}{0}  0  0  0 1 \ .
\]
Now $a_3=1$ has an odd number $m-1 \ge 1$ of neighboring 1's and we are reduced to the previous case. So if $a_3=1$ cannot be changed we must have that all the neighbors of $a_3$ are 0. We now want to push $a_3=1$ to one of the arms. This cannot be done only if each augmented arm -- the subgraph consisting of $a_3$ and the arm itself -- is unmoveable.
\textbf{[}Otherwise, we consider one of the moveable arms: if $a_i$ has an even number of neighboring 1's (other than $a_3=1$) we can push the 1 in $a_3$ into to it and we won. So we must have that the moveable arm is such that $a_i$ has an odd number $m \ge 1$ of neighboring 1's (other than $a_3=1$). We choose the closest moveable 1 to $a_i$ and move it backwards (farther than $a_3$), and then the next (closer) component become moveable, we move it backwards and repeat this process until we move backwards the component closest to $a_i$. Then $m-1$ is even and thus $a_3=1$ can be pushed to $a_i=0$.\textbf{]}
Thus we get that all the augmented arms are unmoveable and thus the labeling is fixed.
\end{proof}

\begin{lemma}[Clear Path Lemma] \label{lem:clear-path}
A {\em clear path arm} is an arm in which we can push the closest component to $a_3$ to the direct neighbor of $a_3$ on that arm when $a_3=0$.
Then any non-fixed labeling is equivalent to a labeling in which each arm is a clear path arm.
\end{lemma}

For example: let us denote label of the vertex $a_3$ by a bold number. We start with a labeling in which the long left arm is not a clear path,
\[
101\two{0}{1}0\two{\boldsymbol{0}}{1}01 \longmapsto 101\two{0}{1}1\two{\boldsymbol{1}}{1}01 \longmapsto 101\two{1}{1}1\two{\boldsymbol{1}}{1}01 \longmapsto 100\two{1}{0}1\two{\boldsymbol{1}}{1}01 \longmapsto 100\two{0}{0}0\two{\boldsymbol{0}}{1}01
\]
and end with a labeling in which all the arms are clear paths, as we can push the leftmost component and the rightmost component while the bottom component is already placed in a direct neighbor of $a_3$.
\[
100\two{0}{0}0\two{\boldsymbol{0}}{1}01 \longmapsto 000\two{0}{0}1\two{\boldsymbol{0}}{1}10
\]
Now all the closest components are placed in the direct neighbors of $a_3$ in each arm.

\begin{proof}[Proof of the Clear Path Lemma]
By Lemma \ref{lem:central-vertex} we may assume we start with a labeling $\aa$ such that $a_3=0$ and we change it to $a'_3=1$ when necessary.

In each arm we consider the closest component to $a_3$ (that is, the component with minimal of number of edges connecting between the vertex it is on to the vertex $a_3$). We say that a nonempty arm is a {\em clear path arm} if we can push the closest component to $a_3$ in that arm to the direct neighbor of $a_3$ in that arm. If the closets component to $a_3$ in the arm is a direct neighbor of $a_3$ then the arm poses no problem in the reduction step and it is a clear path arm.

Let us look on one arm and let $a_i=1$ be the vertex with component closest to $a_3$ on that arm (the arm is not necessarily $\AA_n$).
Then all the labels on the path between the closest component $a_i$ to $a_3$ are 0's. We expand the component towards the vertex $a_3$ until the neighbor of $a_3$ on that arm is labeled 1.
This cannot be done only if there is a branching vertex (i.e. a vertex of degree $d \ge 3$) in the path, call it $a_j$, such that an odd number $m < d$ of its neighbors has 1's in them (call them $a_{j+1},...,a_{j+m}$). We may assume that the distance (i.e. number of edges) between $a_{j+t}$ ($t=1,...,m$) to $a_3$ is equal to the distance between $a_i$ to $a_3$ otherwise any $a_{j+t}$ would be the closest component on that arm. We therefore have that the vertex $a_j$ has an even number $m+1$ of neighbors (including $a_i=1$) with 1's and we call them $a_{j-1}=a_i=1$ and $a_{j+1} = 1, ..., a_{j+m} = 1$.
For example, if $a_j$ is a vertex of degree 3 we have
\[
...a_i \ll \four{a_j}{|}{a_{j+1}}{\vdots} \ll 0...0 \ll a_3  = ...1\three{0}{1}{\vdots}0...0 \ll a_3
\]
If we can move the 1 in (at least) one of the $a_{j+t}$ backwards (e.g. when $a_{j+t}$ is moveable) then we can expand $a_i$ towards $a_3$. Indeed, without loss of generality suppose we can push $a_{j+m}=1$ backwards so $a_{j+m}=0$. Then
\[ a'_j = a_j + a_i + a_{j+1} + ... + a_{j+m} = a_j + 1 + (m-1) + 0 = a_j + m = a_j + 1 \]
as $m \bmod 2 = 1$ is odd. Then we do unsplitting at $a_j$, shrink the merged component so that $a_j=1$ and all of its neighbors are 0's and we have $a_j=1$ and all the vertices on the path between $a_j$ to $a_3$ are 0. Thus the arm becomes a clear path. \newline
If we cannot retreat the 1 in all the $a_{j+t}$ ($t=1,...,d$) or the 1 in $a_i$ farther than $a_3$ we get that the arm is unmoveable. If all the arms are unmoveable then the labeling is fixed. Therefore we assume that at least one arm is moveable.


So we now assume that ar least one arm is moveable. The moveable arm is an arm with clear path. If $a_3=1$ or an odd numbr of neighbors of $a_3$ are labeled 1 we proceed as below or as in case (i) in Part I of the main proof below. If $a_3=0$ we change it to $a'_3=1$ by Lemma \ref{lem:central-vertex}.
Therefore we can set $a'_3=1$ and expand the 1's to the unmoveable arm(s). Then in each unmovable arm we do the following sequence of moves (here the bold 1 denote it is the vertex $a_3$, the subscript $j$ on the branching vertex is to remind this is the vertex $a_j$, and $1_m$ denote the $m$ neighbors of $a_j$ with 1's, recall that $m$ is odd):
\[
...1\two{0_j}{1_m}0...0\boldsymbol{1} \longmapsto ...1\two{0_j}{1_m}1...1\boldsymbol{1} \longmapsto ...1\two{1_j}{1_m}1...1\boldsymbol{1} \longmapsto ...0\two{1_j}{0}1...1\boldsymbol{1} \longmapsto ...0\two{0}{0}0...1\boldsymbol{1}
\]
(the last step can be done as the labeling is not fixed and thus shrinking of components is always possible). Thus we reduce the number of components  with a non-fixed labeling. We repeat this process until there is a clear path between $a_{i+1}$ to $a_3$ with only 0 inbetween and without neighboring 1's in branching vertices along that path. We do this process to each unmoveable arm.

$\implies$ So we may assume for a non-fixed labeling that each arm has a ``clear path'' of 0's from its closest component to $a_3$, without 1's neighboring branching vertices along the path. Therefore, if the arm is not empty then the nearest component can be pushed to the direct neighbor of $a_3$ on that arm.
\end{proof}

\begin{lemma} \label{lem:1-2-comps-are-classes}
In a graph containing $\EE_6$ as a subgraph:
\begin{enumerate}
\item All the nonfixed labelings with 1 component are equivalent to each other and form the class $C_1$.
\item All the nonfixed labelings with 2 components are equivalent to each other and form the class $C_2$.
\end{enumerate}
\end{lemma}
\begin{proof}
Assume that the labeling is not fixed. In this case, in order to conserve the number of components, we do only moves which do not induce splitting, and then we have the set of allowed moves acting as in a $\AA_n$-type diagram. The assertion then follows from Theorem \ref{th:basic-An}.
\end{proof}

\begin{proof}[Proof of Theorem \ref{th:three-arms}]

In the first part of the proof we assume that the branching vertex coming from $\EE_6$ is exactly of degree $d=3$. In the second part we generalize to any degree $d$ of that vertex. In both part I and part II we shall denote this vertex and its label by $a_3$.

\medskip

\textbf{Part I:}

To simplify the proof we introduce the following terminology: let $L$ denote the longest arm with length $l \ge 2$, let $S$ denote the shorter arm with length $s \ge 2$ and let $U$ denote an arm with length $u \ge 1$ ($U$ stands for unsplitting arm). We shall denote the unsplitting vertex of degree 3 by $a_3$, its neighbor at $L$ (resp. $S$, resp. $U$) by $a_L$ (resp. $a_S$, resp. $a_U$). So schematically we have
\[
\sxymatrix{ L \rline & a_3 \dline \rline & S \\ & U & } \ , \qquad
 \sxymatrix{ \cdots \rline & a_L \rline & a_3 \dline \rline & a_S \rline & \cdots \\ & & a_U \dline  & & \\ & & \vdots  & & } \ .
\]
If an arm has no components, i.e. all the labels of its vertices are 0, we say that the arm is empty.

The proof is done by induction on the number of components. If there are 1 or 2 components there is nothing to prove (this is the \textbf{base step of the induction}) and by Lemma \ref{lem:1-2-comps-are-classes} all the labelings with 1 component are equivalent and all nonfixed labelings with 2 components are equivalent. Assume we have $r > 2$ components and assume we proved the assertion that the number of components can be reduced to 1 or 2 (depending on parity) for $r-1$ and $r-2$ components. Recall that moves preserve parity so if we manage to reduce the number of components by 2 we've done \textbf{the induction step} and then invoke the induction hypothesis on $r-2 < r-1 < r$ components. So we need to show that in all the cases of labelings which are not fixed labelings we can reduce the number of components by 2.

\textbf{Claim:} If a labeling in not a fixed labeling with $r \ge 3 > 2$ components, then we can reduce the number of components by 2 (by unsplitting at $a_3$).

To prove this claim we consider all possible cases. Given a (non-fixed) labeling we may assume that all the components are of minimal size 1 (in non-fixed labeling we can always change it to be of that state). By Lemma \ref{lem:clear-path} we may and will assume that in the non-fixed labelings we treat all the nonempty arms are clear path arms.
We now prove the claim by checking all possible cases.

\textbf{Case (i):} We assume each arm has at least one component and that the labeling is not a fixed labeling. If the labeling is not fixed we can change $a_3$ to be 0 by Lemma \ref{lem:central-vertex}. So we may assume $a_3=0$. By Lemma \ref{lem:clear-path} each arm is a clear path arm and since each arm is not empty then we can push the closest components to the vertices $a_L, a_S, a_U$ and then do unsplitting $\T_3$ as in Example \ref{exmp:pushing1}.

\textbf{Case (ii):} Assume $U$ is empty while $L$ and $S$ are not, and the labeling is not a fixed labeling. Assume $L$ has $p$ components and $S$ has $q$ components and (without loss of generality) $p > q$. If $a_3 = 1$ we can change the label $a_3$ to $a_3=0$ by Lemma \ref{lem:central-vertex}. We may assume $a_L=a_S=0$ (otherwise we swallow the closest 1's in each arm). Then apply $\T_U$ and $\T_3$. Then $U$ has 1 component and we proceed as in the previous case. So we now assume $a_3 = 0$. If $a_S=0$ we can move a component from $L$ to $U$. If $a_S=1$ we want to push the 1 in $a_S$ further into $S$ (farther than $a_3$ so $a_S=0$). This cannot be done only if $a_S=1$ and $S$ is unmoveable. If $a_L=0$ we can move a component from $S$ to $U$. This cannot be done if $a_L=1$ and $L$ is unmoveable. The empty $U$ is clearly unmoveable so we can't move components from either $L$ or $S$ only if the labeling is fixed, which we assume is not the case. Therefore, by Lemma \ref{lem:clear-path} we can push a component from some arm to $U$. We now have at least 1 component in each arm (if $S$ is unmoveable then $q \ge 2$ so $q-1 \ge 1$) so we are reduced to case (i).
See Example \ref{exmp:pushing2}.

\textbf{Case (iii):} Assume only $S$ is empty and that $U$ has at least $v \ge 1$ components and $L$ has $p \ge 2$ components. By Lemma \ref{lem:central-vertex} we may assume $a_3=0$. If the labeling is not fixed one can move a component from $U$ to $S$. Indeed, the only way not to be able to push a component from $U$ to $S$ is the to have the labeling
\[
...1 \three{0}{1}{\vdots}0...0
\]
with $L$ being unmoveable. In that case we shall move a component from $L$ to $S$ but this cannot be done only if the labeling is with $a_U=1$ and $U$ is unmoveable. So we get that we cannot do the reduction step only if $L$ and $U$ are unmoveable. But this is a fixed labeling (as $S$ is unmoveable per being empty), so the only case when we cannot perform the algorithm is excluded by our assumptions.
\[ ( \mbox{Example: if $U$ is of length 1 then we have fixed labeling} \quad  101...01 \two{0}{1}0...0 \ ) \]
If $U$ is still not empty, we can do unsplitting at $a_3$ and reduce the number of components by 2. If $U$ is now empty, we take the component $v$ and push it away from $a_S$ (this can be done since $S$ is longer than 2 and was empty). We then move the closest component from $L$ to $U$ (possible due to Lemma \ref{lem:clear-path}). Now we have $p-1 \ge 1$ components in $L$, 1 component in $S$ and 1 component in $U$. So we are again reduced to case (i).
See Example \ref{exmp:pushing4}.

\textbf{Case (iv):} the case where only $L$ is empty is handled the same way as the previous case.

\textbf{Case (v):} two arms are empty, without loss of generality assume $L$ has $p \ge 3 > 2$ components. By Lemma \ref{lem:central-vertex} we may assume $a_3=0$. We move one component from $L$ to $S$ and push it so that $a_S=0$. We move another component from $L$ to $U$. We have $p-2$ components in $L$, 1 component in $S$ and 1 component in $U$. We have at least one component in each arm so we are reduced again to case (i).
See Example \ref{exmp:pushing3}.

The cases (i)--(v) exhaust all the possible cases. In each case we manage to reduce the number of components by 2 and then invoke the induction hypothesis. Actually, the algorithm hidden in the induction formulation is applying repeatedly the algorithms of Examples \ref{exmp:pushing1}--\ref{exmp:pushing4} until we are left with 1 or 2 components. Then clearly no more unsplitting can be done, and by Lemma \ref{lem:1-2-comps-are-classes} we get to the equivalence class of $C_1$ or $C_2$ depending on parity. This concludes the proof of part I.

\medskip

Let us illustrate the reduction step (induction step) of the algorithm of the proof in four simple examples.

\begin{example} \label{exmp:pushing1}
Consider the diagram
\[
\sxymatrix{ \bc{1} \rline & \bc{2} \rline &  \bc{3} \rline & \bc{4} \rline  & \bc{5} \rline & \bc{6} \dline \rline & \bc{7} \rline & \bc{8} \\
                           &              &                &                &               & \bcu{9}        &               &       } ,
\]
and we start from the labaling $10000\two{0}{1}10$. Then
\[
10000\two{0}{1}10 \longmapsto 1111\two{0}{1}10 \overset{\T_6}{\longmapsto} 11111\two{1}{1}10 \overset{\T_6\T_4}{\longmapsto} 11111\two{1}{0}00 \overset{\T_6}{\longmapsto} 11111\two{0}{0}00 \longmapsto 10000\two{0}{0}00
\]
and we are reduced to one component from three.
\end{example}

\begin{example} \label{exmp:pushing2}
Now consider the same diagram with labeling $10100\two{0}{0}10$. Then
\[
10100\two{0}{0}10 \mapsto 10100\two{0}{0}01 \mapsto 10000\two{0}{1}01 \mapsto 00001\two{0}{1}10 \mapsto 00001\two{1}{1}10 \mapsto 00000\two{1}{0}00
\]
and we are reduced to one component from three.
\end{example}

\begin{example} \label{exmp:pushing3}
Now consider the same diagram with labeling $10101\two{0}{0}00$. Then
\[
10101\two{0}{0}00 \mapsto 10100\two{0}{0}01 \mapsto 10000\two{0}{1}01 \mapsto 00001\two{0}{1}10 \mapsto 00001\two{1}{1}10 \mapsto 00000\two{1}{0}00
\]
and we are reduced to one component from three.
\end{example}

\begin{example} \label{exmp:pushing4}
Now consider the same diagram with labeling $10100\two{0}{1}00$. Then
\[
10100\two{0}{1}00 \mapsto 10100\two{0}{0}01 \mapsto 10100\two{0}{0}01 \mapsto 10000\two{0}{1}01 \mapsto 00001\two{0}{1}10 \mapsto 00001\two{1}{1}10
\]
and we are reduced to one component from three.
\end{example}

\bigskip

\textbf{Part II:}

Assume now that the branching vertex coming from $\EE_6$ (we shall call it again $a_3$) is of degree $d > 3$. It has $d$ arms, at least two of them are in length $\ge 2$. As in Part I we use induction on the number of components $r$ to show every non-fixed labeling is equivalent to a labeling with 1 or 2 components depending on the initial parity. The base step is $r=1,2$ and then we have nothing to prove, thanks to Lemma \ref{lem:1-2-comps-are-classes}. The induction step is to reduce the number of components by 2 or more.

We assume that the labeling $\aa$ is not fixed and that $r \ge 3$. By Lemma \ref{lem:clear-path} we may assume that all the arms are clear paths and moveable.
\begin{itemize}
\item If $a_3=0$ and $a_3$ has an odd number $h_1 \ge 3 > 1$ of neighbors with 1's we do unsplitting at $a_3$ and reduced the number of components by $h_1-1$.
\item If $a_3=0$ and $a_3$ has only one neighbor with 1 in it we choose the arm with this 1 and two more arms with length $\ge 2$ such that together they have 3 or more components. Otherwise, $1 \le r \le 2$. We then use part I of the proof on the subgraph consisting of $a_3$ and the chosen 3 arms (we can do it as all the other neighbors of $a_3=0$ are 0 and thus $\T_{j,\mbox{full}} = \T_{j,\mbox{subgraph}}$ for all vertices $j$ in the subgraph, including $a_3$). This is the reduction step.
\item If $a_3=0$ has an even number $h_2 \ge 4 > 3$ (as we assume $r > 3$) of neighbors with 1's. In this case if one of the 1's can be moved backwards (i.e. farther than $a_3$ on its arm) we are reduced to the previous case with $h_1 = h_2 - 1$ being odd. If none of the 1's can be moved backwards then all the arms are unmoveable and since $\aa = \T_3(\aa)$ (this is because $a_3=0$ and it has an even number of neighbors with 1) we get that $\aa$ if fixed, which is excluded by our assumption.
\item If $a_3=0$ and all of its neighbors are 0's we choose 3 arms such that at least 2 of them are of length $\ge 2$ and together the 3 arms have at least 3 components. We can choose such arms since otherwise the total number of components is $0 \le r \le 2$. We then use part I of the proof on the subgraph consisting of $a_3$ and the chosen 3 arms (we can do it as all the other neighbors of $a_3=0$ are 0 and thus $\T_{j,\mbox{full}} = \T_{j,\mbox{subgraph}}$ for all vertices $j$ in the subgraph, including $a_3$). This is the reduction step.
\item If $a_3=1$ and the labeling is not fixed then by Lemma \ref{lem:central-vertex} it is equivalent to a labeling $\aa'$ with $a'_3=0$ and we are reduced to the previous case of the central label being 0.
\end{itemize}
Thus we proved that if the labeling $\aa$ is not fixed we can reduce its number of component by an even number until we are left with 1 or 2 components depending on the initial parity.

\end{proof}


\begin{corollary}
The diagrams $\EE_n$ and $\tilde{\EE}_n$ for $n=6,7,8$ satisfy the condition of Theorem \ref{th:three-arms} and therefore they have only 2 classes which are not classes of fixed labelings. These classes are the class of labelings with odd number of components which are not fixed labelings and the class of labelings with even number of components which are not fixed labelings.
\end{corollary}

The generalization to the cases when the diagram is not simply-laced is not difficult, but one should consider many cases. For example, if one of the ``tails'' is $\BB_n$ then the shorter vertex added another fixed labeling. If one of the ``tails'' is $\CC_n$ then there are 4 non-fixed-points classes: the class with odd number of components and the longer vertex is labeled 0, the class with even number of components and the longer vertex is labeled 0, the class with odd number of components and the longer vertex is labeled 1 and the class with even number of components and the longer vertex is labeled 1.

\bigskip

So far we only treated vertices of degree $1 \le d \le 3$. Let us treat general vertices of any degree.

\begin{proposition} \label{prop:general-d-parity}
Let $D$ be a simply-laced tree with a vertex $i$ of degree $d$ (i.e. it has $d$ neighbors). Then the move $\T_i$ preserve the parity of the number of components. More specifically, assume the vertex $i$ has $0 \le m \le d$ neighbors with label 1. That is: if $a_i=0$ then the subgraph of the vertex $i$ and its $d$ neighbors $D(i)$ has $m$ components. If $m$ is even then $\T_i$ does not change the number of components. If $m = 2k+1$ is odd then $\T_i$ change the number of components by $\pm (m-1) = \pm 2k$ and hence preserve the parity of the number of components.
\end{proposition}

\begin{proof}
If $\aa' = \T_i(\aa)$ then
\[ a'_i = a_i + \sum_j a_j = a_i + m \ (\bmod 2) \ . \]

If $m = 2k$ is even then $a'_i = a_i$ and thus $\aa' = \aa$ and the parity of the number of components is preserved.

If $m=2k+1$ is odd and $\aa' = \T_i(\aa)$ then $a'_i =  a_i + 1$. If $a_i = 0$ then the subgraph of the vertex $i$ and its $n$ neighbors has $m$ components, but after unsplitting it has only 1 component (the 1 in the vertex $i$ merges all the components into a single component). Thus the number of components has changed by $\Delta = 1 - m = -2k$ which is an even number. For the inverse process, splitting, the change is $\Delta = m - 1 = 2k$ which is even. Thus we see that the parity of the number of components is preserved.
\end{proof}

%

\begin{definition} \label{def:flower}
Let $\maltese_d$ be the a simply-laced graph with $d+1$ vertices: one vertex of degree $d$ (we number it by 0) which have $d$ neighbors (numbered $1,2,...,d$), each of degree 1, called {\em petals}. The graph $\maltese_d$ is called a {\em flower diagram} or a {\em flower graph} with $d+1$ vertices and $d$ petals.
\end{definition}


\begin{proposition} \label{prop:d-vertex}
Let $\maltese_d$ be a flower diagram.  Then the possible numbers of components for the graph $\maltese_d$ can be $0$ to $d$. If $\maltese_d$ has a labeling with 1 component then it is equivalent to any labeling with $m \le d$ components where $m$ is odd.
\end{proposition}

\begin{proof}
Let $\maltese_d$ be the flower graph with central vertex numbered 0. Assume we have 1 component. By a sequence of moves we can transform to a labeling with $a_i = 1$. Then we swallow all the 1's in the arms. We now have $a_i=1$ and $a_j = 0$ for every neighbor $j$ of $i$. Now we spawn 1's in $m$ neighbors of $i$ where $m$ is odd. Since $m$ is odd, by Proposition \ref{prop:general-d-parity}, we can do unsplitting and now we have $m$ components. The inverse process is clear. This is true for any odd $m$ and we get that every labeling with odd number of components in $D(i)$ is equivalent to the labeling with 1 component. Since this equivalence is an equivalence relation we get that all the labeling with odd number of components are equivalent to each other.
\end{proof}

\begin{lemma}[Well known]
\label{lem:binom1}
\begin{equation*}
\sum_{0 \le 2m \le d} {d \choose 2m} = \sum_{0 \le 2m+1 \le d} {d \choose {2m+1}} = 2^{d-1} \ .
\end{equation*}
\end{lemma}

\begin{proof}
Denote $\lambda = \sum_{0 \le 2m \le d} {d \choose 2m}$.
Consider the binomial expansion
\[ (x + y)^d = \sum_{k=0}^d {d \choose k} x^k y^{d-k} \ . \]
We have
\[  0 = (-1 + 1)^d = \sum_{k=0}^d {d \choose k} (-1)^k = \sum_{0 \le 2m \le d} {d \choose 2m} - \sum_{0 \le 2m+1 \le d} {d \choose {2m+1}} \]
and thus
\[ \sum_{0 \le 2m+1 \le d} {d \choose {2m+1}} = \sum_{0 \le 2m \le d} {d \choose 2m} = \lambda   \ . \]
But
\[ 2^d = (1 + 1)^d =  \sum_{k=0}^d {d \choose k} =  \lambda + \lambda = 2 \lambda  \]
and thus $\lambda = 2^d/2 = 2^{d-1}$.
\end{proof}

\begin{theorem}[The Flower Diagram Theorem]
 The diagram $\maltese_d$ has the following classes:
\begin{itemize}
\item The zero class of the zero labeling (i.e. all labels are 0).
\item For every $0 < 2m \le d$ we have ${d \choose 2m}$ fixed labelings, each with $2m$ components. These labelings are all combinations of putting $2m$ components on the $d$ petals (and the central vertex's label is 0).
\item The class of all labelings which have odd number of components, represented by the labeling with central vertex with label 1 and petal vertices with labels 0.
\end{itemize}
The number of classes is
\begin{equation} \label{eq:flower-number}
\Cls{\maltese_d} = 1 + {d \choose 2} + ... + {d \choose { 2 \lfloor d/2 \rfloor} } + 1 = 2^{d-1} + 1
\end{equation}
\end{theorem}

\begin{proof}
By Proposition \ref{prop:d-vertex} all labelings with odd number of components are equivalent. This accounts for the single class which is listed last in the theorem.

Now let $\aa$ be a labeling with $2m$ components and with $a_0=0$. Let us number the petals $a_1,...,a_d$. Clearly, applying a move on each petal does not change it as its only neighbor, $a_0$, is 0. It remains to check that $\T_0(\aa)=\aa$, but this is clear since
\[ \aa'_0 = (\T_0(\aa))_0 = a_0 + \sum_{i=1}^d a_i = a_0 + 2m \bmod 2 = a_0 = (\aa)_0 \ . \]
Thus, a labeling with even number of components is fixed.
How many such labelings do we have? The answer here is combinatoric: choose $2m$ petals out of the possible $d$ petals to put 1's in, there are ${d \choose 2m}$ such options.

Finally, we run over the number of all possible even numbers smaller than $d$ and thus we account of all the classes. This leads to \eqref{eq:flower-number} -- the leftmost 1 is $1 = {d \choose 0}$ and the rightmost 1 is the class with odd number of components. The rightmost equality in \eqref{eq:flower-number} follows from Lemma \ref{lem:binom1}.
\end{proof}

\begin{example}
Consider the graph $\maltese_4$ given by
\begin{equation}
\sxymatrix{ & \bc{1} \dline & \\ \bc{2} \rline & \bc{0} \dline \rline & \bc{3} \\ & \bcu{4} & }
\end{equation}
Then its has 9 classes as follows:
\begin{enumerate}
\item[(1)] The fixed labeling $$ \sxymatrix{ & 0 \dline & \\ 0 \rline & 0 \dline \rline & 0 \\ & 0 & } $$
\item[(2)-(7)] The following 6 fixed labelings
\begin{eqnarray*}
 \sxymatrix{ & 1 \dline & \\ 1 \rline & 0 \dline \rline & 0 \\ & 0 & }  \quad
 \sxymatrix{ & 1 \dline & \\ 0 \rline & 0 \dline \rline & 1 \\ & 0 & }  \quad
 \sxymatrix{ & 1 \dline & \\ 0 \rline & 0 \dline \rline & 0 \\ & 1 & }  \quad
 \sxymatrix{ & 0 \dline & \\ 1 \rline & 0 \dline \rline & 1 \\ & 0 & }  \quad
 \sxymatrix{ & 0 \dline & \\ 1 \rline & 0 \dline \rline & 0 \\ & 1 & }  \quad
 \sxymatrix{ & 0 \dline & \\ 0 \rline & 0 \dline \rline & 1 \\ & 1 & }
\end{eqnarray*}
\item[(8)] The fixed labeling $$\sxymatrix{ & 1 \dline & \\ 1 \rline & 0 \dline \rline & 1 \\ & 1 & }$$ with 4 components
\item[(9)] The class represented by $$\sxymatrix{ & 0 \dline & \\ 0 \rline & 1 \dline \rline & 0 \\ & 0 & }$$ of all labelings which have 1 or 3 components.
\end{enumerate}

\end{example}

\end{document}